\documentclass[12pt]{article}
\usepackage[fontsize=11.8pt]{scrextend}

\title{High-dimensional tennis balls}
\author{W. T. Gowers \and K. Wyczesany}

\date{}

\usepackage[utf8]{inputenc}

\usepackage{amsmath, wrapfig}
\usepackage{dsfont, a4wide, amsthm, amssymb, amsfonts, graphicx}
\usepackage{fancyhdr, xspace, psfrag, setspace, supertabular, color}
\usepackage[colorlinks]{hyperref}
\usepackage{bbm}
\usepackage{stmaryrd}

\usepackage{float}
\usepackage{wrapfig}
\usepackage{verbatim}
\usepackage{pgf,tikz}
\usetikzlibrary{arrows.meta,calc,decorations.markings,math,arrows.meta}
\usetikzlibrary{positioning}
\usetikzlibrary{hobby}
\usetikzlibrary{shapes,arrows,chains}
\usetikzlibrary[calc]
\usepackage[margin=2.5cm]{geometry}

\newtheorem{theorem}{Theorem}[section]
\newtheorem{proposition}[theorem]{Proposition}
\newtheorem{lemma}[theorem]{Lemma}
\newtheorem*{claim1*}{Claim 1}
\newtheorem*{claim2*}{Claim 2}
\newtheorem*{definition*}{Definition}
\newtheorem{corollary}[theorem]{Corollary}

\newtheorem{question}[theorem]{Question}

\theoremstyle{definition}
\newtheorem{remark}[theorem]{Remark}
\newtheorem{definition}[theorem]{Definition}

\def\e{\epsilon}
\def\E{{\mathbb{E}}}

\def\R{\mathbb{R}}

\def\N{\mathbb{N}}
\def\P{\mathbb{P}}

\def\a{\alpha}
\def\be{\beta}

\def\g{\gamma}
\def\d{\delta}

\def\v{\varphi}

\def\D{\Delta}
\def\b1{\mathbbm{1}}

\newcommand{\appropto}{\mathrel{\vcenter{
  \offinterlineskip\halign{\hfil$##$\cr
    \propto\cr\noalign{\kern2pt}\sim\cr\noalign{\kern-2pt}}}}}
    
\newcommand{\vertiii}[1]{{\left\vert\kern-0.25ex\left\vert\kern-0.25ex\left\vert #1 
    \right\vert\kern-0.25ex\right\vert\kern-0.25ex\right\vert}}

\def \sp#1{\langle #1\rangle}

\def\kasia#1{{\color{magenta} \textsc{(Kasia says: } \textsf{#1})}}

\begin{document}

\maketitle

\begin{abstract}
We show that there exist constants $\a,\,\e>0$ such that for every positive integer $n$ there is a continuous odd function $f:S^m\to S^n$, with $m\geq \a n$, such that the $\e$-expansion of the image of $f$ does not contain a great circle. This result is motivated by a conjecture of Vitali Milman about well-complemented almost Euclidean subspaces of spaces uniformly isomorphic to $\ell_2^n$.    
\end{abstract}

\section{Introduction}
Let $A$ be a measurable subset of the sphere $S^n$. How large must the measure of $A$ be in order to guarantee that the $\e$-expansion of $A$, that is, the set $A_\e$ that consists of all points at distance at most $\e$ from $A$, contains the unit sphere of a subspace of dimension $k$? Such questions have been much studied ever since Milman's famous proof \cite{Milman1971} of Dvoretzky's theorem \cite{Dvoretzky1961}. Milman's insight was that by the isoperimetric inequality in the sphere, the volume of the $\e/2$-expansion (say) is minimized, for a given measure of $A$, when $A$ is a spherical cap. But when $A$ is a spherical cap of measure $\a$ and $n$ is large, a relatively straightforward calculation shows that $A_{\e/2}$, which is again a spherical cap, has measure very close to 1. From this it follows, again straightforwardly, that under suitable conditions on the parameters, an $\e/2$-net of the sphere of a \emph{random} subspace of dimension $k$ will lie in $A_{\e/2}$ with high probability, and hence that the entire sphere will lie in $A_\e$. This basic argument can be used to prove the surprising result that even if $\e$ is quite small, an exponentially small measure for $A$ is sufficient to guarantee that $A_\e$ contains the sphere of an $k$-dimensional subspace for some $k$ of dimension that is linear in $n$.

If $A$ has measure $c^n$ and $c$ is too small, then the above argument fails, and the conclusion is false. Indeed, if $A$ is a spherical cap of volume $c^n$, and if its spherical radius is $\pi/2-\eta$, then unless $\e>\eta$ the volume of $A_\e$ will not be close to 1, and if $\e<\eta$ then $A_\e$ will not even contain two antipodal points, let alone the sphere of a subspace of dimension $k$.

One could attempt to rule out this simple example by restricting attention to centrally symmetric sets, but that does not achieve much: if $A$ is the cap just discussed, and if its centre is the unit vector $u$, then $A\cup(-A)$ is centrally symmetric, and $(A\cup(-A))_\e$ is disjoint from the hyperplane orthogonal to $u$, which implies that $(A\cup(-A))_\e$ does not contain the sphere of any 2-dimensional subspace. 

A noticeable feature of this example is that it is in a certain sense zero-dimensional: if we identify antipodal points and write $q$ for the quotient map, then $q(A)$ is a subset of projective $n$-space, and it is homotopic to a point. Having made this observation, it is natural to wonder what happens if we impose a condition that forces $q(A)$ to have a higher dimension in this topological sense. That motivates the following definition.

\begin{definition} An $m$-dimensional \emph{topological subsphere} of $S^n$ is the image of a continuous function $f:S^m\to S^n$ that preserves antipodal points. \end{definition}

\noindent If an $m$-dimensional topological subsphere of $S^n$ is the unit sphere of an $(m+1)$-dimensional subspace of $\R^{n+1}$, then we shall call it \emph{linear}.

An $m$-dimensional topological subsphere is in a certain sense ``genuinely $m$-dimensional". For instance, if $X$ is such a subsphere and $g:X\to\R^m$ is a continuous function, then $g\circ f$ is a continuous function from $S^m$ to $\R^m$, which implies, by the Borsuk-Ulam theorem, that there is some $x\in S^m$ such that $g(f(x))=g(f(-x))$, and therefore that $g(f(x))=g(-f(x))$. Thus, for any continuous map from $X$ to $\R^m$ there will be two antipodal points with the same image.

We now ask the following question. 

\begin{question} Let $\e>0$, let $k$ be a positive integer, and let $X$ be an $m$-dimensional topological subsphere of $S^n$. How large does $m$ have to be in order to guarantee that $X_\e$ contains a linear subsphere of dimension~$k$?\end{question}

In order to tackle this question, an obvious first step is to see how well one can do using concentration of measure. That is, we consider instead a slightly stronger question.

\begin{question} Let $\e>0$, let $k$ be a positive integer, and let $X$ be an $m$-dimensional topological subsphere of $S^n$. How large does $m$ have to be in order to guarantee that $X_\e$ contains almost all linear subspheres of dimension $k$?\end{question}

By standard arguments, that is roughly the same as asking for $X_\e$ to have measure at least $1-\e^k$.

The following estimate is well known. See for example \cite{Artstein2002}.

\begin{lemma} Let $m=\a n$ and let $\e(\a)$ be such that $\sin\e(\a)=\sqrt{1-\a}$. Then if $X$ is a linear $m$-dimensional subspace, the measure of $X_\e$ tends to 1 if $\e>\e(\a)$ and to 0 if $\e<\e(\a)$. 
\end{lemma}

\noindent This implies that for the second question we need $m$ to be at least $\a n$, where $\sqrt{1-\a}=\sin\e$, or $\a=\cos^2\e\approx 1-\e^2/2$. 

However, it is not obvious what this observation tells us about the first question, since these examples are linear subspheres, which are good sets to choose for the second question but the worst possible sets to choose for the first. That is, if we wish to find a topological subsphere $X$ of dimension $m$ such that $X_\e$ contains only a very small proportion of all linear $k$-dimensional subspaces, then we should take $X$ itself to be linear, but if we would like $X_\e$ to contain \emph{no} linear $k$-dimensional subspace, then obviously we cannot take $X$ to be linear (unless its dimension is less than $k$). 

The main result of this paper is that even when $k=1$, the dimension of $X$ can be quite large.

\begin{theorem} \label{tennisball} There exist constants $\a,\e>0$ such that for every $n$ there is an $\lfloor\a n\rfloor$-dimensional topological subsphere $X$ of $S^n$ such that $X_\e$ contains no linear subsphere of dimension 1. \end{theorem}

\noindent To put this less formally, there is a topological subsphere $X\subset S^n$ of dimension linear in $n$ such that the expansion $X_\e$ does not contain any 1-dimensional subsphere. 

We informally call such a topological subspace a \emph{tennis ball} because it brings to mind the seam of a genuine tennis ball (though the resemblance is not perfect, since the seam of a genuine tennis ball is not centrally symmetric). We shall also refer to 1-dimensional subspheres as \emph{great circles}.

Theorem \ref{tennisball} has a simple corollary that can be thought of as an anti-Ramsey theorem for linear subspheres.

\begin{corollary}\label{antiramsey}
There exist constants $\alpha,\eta>0$ such that for every $n$ there is a partition of $S^n$ into two subsets $R$ and $B$ such that $R_\eta$ does not contain the unit sphere of any subspace of codimension less than $\lfloor\alpha n\rfloor$ and $B_\eta$ does not contain the unit sphere of any 2-dimensional subspace.
\end{corollary}

\begin{proof}[Proof (assuming Theorem \ref{tennisball})]
Let $\a, \e$  and $X$ be as given by Theorem \ref{tennisball}.  Let $B=X_{\e/2}$,  and $R$ its complement on $S^n$.  Then since $X_\e$ does not contain a great circle,  every great circle contains a point that does not belong to $(X_{\e/2})_{\e/2}=B_{\e/2}$,  which implies that this point is at distance at least $\e/2$ from $B$.

It remains to prove that if $R=S^n\setminus B$ and $r< m$,  where $m=\lfloor \alpha n\rfloor$,  then $R_{\e/2}$ does not contain the sphere of a subspace of codimension $r$, since then we will be done with $\eta=\e/2$. Since $R_\e$ is disjoint from $X$, it is sufficient to prove that $X$ intersects every subspace of codimension $r$.

Let $V$ be such a subspace. We are given that $X$ is the image of some continuous odd function $f:S^m\to S^n$.  Hence,  after composing $f$ with a rotation, we may assume without loss of generality that $V=\{x:x_1=\dots=x_r=0\}$.  Let $P_r$ be the coordinate projection to the first r coordinates.  Then $P_r\circ f$ is a continuous map from $S^m$ to $\R^r$,  so by the Borsuk-Ulam theorem there exists $x\in S^m$ such that $P_r f(x)=P_r f(-x)$. Since $P_r\circ f$ is also odd,  it follows that $P_r f(x)=0$,  and therefore that $f(x)\in V$,  as claimed.
\end{proof}

We actually deduce Theorem \ref{tennisball} from a stronger result that trivially implies it and is of independent interest.

\begin{theorem} \label{main} There exist constants $\a,\e>0$ and a continuous map $\psi:S^n\to S^n$ such that $\psi$ preserves antipodal points and such that if $X$ is a random linear subsphere of dimension $\lfloor\a n\rfloor$ then with probability $1-o(1)$ the set $\psi(X)_\e$ does not contain a great circle.\end{theorem}

\iftrue
\else
\subsection{Questions related to a question of Milman}

Milman's proof of Dvoretzky's theorem led to an explosion of activity in the theory of finite-dimensional normed spaces and to many striking results about the subspace structure of a normed space, often with surprisingly weak hypotheses on the space. Most of these results concerned the subspaces and not their relationship with the space itself, but information about the latter can be very helpful. In particular, if $X$ is a space, $Y$ is a subspace of $X$ and there is a projection from $X$ to $Y$ with bounded operator norm, then we can write $X=Y+Z$ and the norm of a vector $y+z$ is approximated to within a constant factor by $\|y\|+\|z\|$. Such subspaces are called \emph{complemented}, or more precisely $C$-\emph{complemented} if the projection has norm at most $C$.

Our knowledge about which subspaces a normed space must have under certain conditions is less satisfactory if we impose the condition that they should be complemented, and there are several open problems. For example, it is not known whether there is a constant $C$ and a function $f:\N\to\N$ that tends to infinity such that every $n$-dimensional normed space has a $C$-complemented subspace of dimension and codimension at least $f(n)$. (For a partial result in this direction, see \cite{szarektomczak2009}.)

Another open problem is the following question of Milman, which was the starting point for the work of this paper. (The question does not seem to have appeared in print, but he has mentioned it in conversation to many people, including the authors of this paper.) We write $|\cdot|$ for the standard Euclidean norm on $\R^n$. Recall that the \emph{Banach-Mazur distance} between two normed spaces $X$ and $Y$ is the infimum of $\|T\|\|T^{-1}\|$ over all invertible continuous linear maps from $X$ to $Y$, or equivalently the infimum over all $C$ such that there exists an invertible linear map $T:X\to Y$ such that $\|x\|\leq\|Tx\|\leq C\|x\|$ for every $x\in X$.

\begin{question}\label{milman} Let $\e>0$, $C\geq 1$ and $k\in\N$. Does there exist $n$ such that if $\|\cdot\|$ is a norm on $\R^n$ such that $|x|\leq\|x\|\leq C|x|$ for every $x\in\R^n$, then the space $(\R^n,\|\cdot\|)$ has a subspace $Y$ of dimension $k$ such that $d(Y,\ell_2^k)\leq 1+\e$ and $Y$ is $(1+\e)$-complemented?\end{question}

Without the requirement that $Y$ should be $(1+\e)$-complemented the answer is yes, and for fixed $\e,C$ the dependence of $n$ on $k$ is linear, as we have already remarked. However, the additional requirement changes everything, and this problem is unsolved even when $k=2$, and even if we do not ask for any quantitative information about $n$. The assumption that the norm is $C$-equivalent to the Euclidean norm is a strong one, but at least some kind of assumption is needed, since for a general $n$-dimensional normed space \textcolor{blue}{it can be shown that Question \ref{milman} has a negative answer.} \kasia{I feel like we were avoiding such a statement... Are you happy with it?}

The proof when $Y$ is not required to be complemented actually yields a stronger statement: not only is $d(Y,\ell_2^n)\leq 1+\e$ but the linear map that witnesses that fact is the identity. That is, we can find $Y$ and a positive real number $\a$ such that $\a|y|\leq\|y\|\leq(1+\e)\a|y|$ for every $y\in Y$. Loosely speaking, we can find a $(k-1)$-dimensional linear subsphere of $S^{n-1}$ on which the function $\|\cdot\|$ is approximately constant. Also loosely speaking, whereas the statement that $d(Y,\ell_2^k)\leq 1+\e$ implies that the unit sphere of $Y$ is approximately a $(k-1)$-dimensional ellipsoid, the proof yields a subspace such that the unit sphere is approximately a $(k-1)$-dimensional sphere. (It is not hard to prove that a $k$-dimensional ellipsoid has a $\lfloor k/2 \rfloor$-dimensional spherical cross section, so this is not a significant strengthening, but Milman's argument gives spherical cross sections directly.)

If the subspace $Y$ is such that $d(Y,\ell_2^k)\leq 1+\e$, then it is called $(1+\e)$-\emph{Euclidean}. If $Y$ has the stronger property, i.e. there is some constant $\alpha>0$ such that $\alpha |y|\leq \|y\|\leq (1+\e)\alpha|y|$ for all $y\in Y$, let us call it \emph{strongly $(1+\e)$-Euclidean}.

We can similarly strengthen the property of being complemented: let us say that a subspace $Y$ of $(\R^n,\|\cdot\|)$ is \emph{strongly $(1+\e)$-complemented} if not just some projection but the orthogonal projection onto $Y$ has norm at most $1+\e$. 

The following question is a natural strengthening of Milman's original question.

\begin{question}\label{strongmilman} Let $\e>0$, $C\geq 1$ and $k\in\N$. Does there exist $n$ such that if $\|.\|$ is a norm on $\R^n$ such that $|x|\leq\|x\|\leq C|x|$ for every $x\in\R^n$, then the space $(\R^n,\|.\|)$ has a subspace $Y$ of dimension $k$ that is strongly $(1+\e)$-Euclidean and strongly $(1+\e)$-complemented?\end{question}

\noindent Our interest in this stronger question is that it seems quite unlikely that Milman's question would have a positive answer unless this question also has a positive answer, and for various reasons this question is a little more approachable.

We shall now give a simple reformulation of the condition that $Y$ is strongly $(1+\e)$-Euclidean and strongly $(1+\e)$-complemented. We begin with another definition. As before, $|\cdot|$ is the standard Euclidean norm on $\R^n$, and $\langle\cdot , \cdot \rangle$ is the standard inner product. 

\begin{definition}\label{egooddef} Let $X=(\R^n,\|\cdot\|)$ be a normed space and let $x\in X$. We say that $x$ is $\e$-\emph{good} if 
\[\langle x,y\rangle\leq(1+\e)\frac{\|y\|}{\|x\|}|x|^2,\]
for every vector $y\in\R^n$.\end{definition}

To see what this means geometrically, consider the orthogonal projection $P_x$ onto the 1-dimensional subspace of $\R^n$ generated by $x$. Writing $x'$ for the normalized vector $x/|x|$, this has the formula
\[P_xy=\langle x',y\rangle x'.\]
Hence, the operator norm of $P_x$ (as a map from $X$ to $X$) is the maximum of the quantity
\[\frac{\langle x',y\rangle\|x'\|}{\|y\|}=\frac{\langle x,y\rangle\|x\|}{|x|^2\,\|y\|}\]
over all $y\in\R^n$. It follows that $x$ is $\e$-good if and only if the orthogonal projection $P_x$ has norm at most $1+\e$ in the space $L(X)$.  Clearly, the definition of $\e$-good point does not depend on the norm of the vector and  hence it is enough to consider unit vectors. Let us write $S^n=\{x\in \R^{n+1}: |x|=1\}.$

We now show that a subspace $Y$ of a space $X$ is strongly $(1+\e)$-Euclidean and strongly $(1+\e)$-complemented for some small $\e$ if and only if every $y\in Y$ is $\d$-good for some small $\d$.

\begin{lemma}\label{lem:good=compl+eucl}
Let $X=(\R^n,\| \cdot \|)$ be a normed space and let $Y\subset X$ be a subspace. Then
\begin{enumerate}
\item if $Y$ is strongly $(1+\e)$-complemented and strongly $(1+\e)$-Euclidean, then every $y\in Y$ is $(2\e+\e^2)$-good;
\item if $\e\leq 1/9\pi^2$ and every point in $Y$ is $\e$-good, then $Y$ is strongly $(1+\e)$-complemented and strongly $(1+3\pi\sqrt{\e})$-Euclidean. 
\end{enumerate}
\end{lemma}

\begin{proof}
    Let $P_Y$ be the orthogonal projection onto $Y$. If $Y$ is strongly $(1+\e)$-Euclidean and strongly $(1+\e)$-complemented, then $\|P_Y x\| \le (1+\epsilon)\|x\|$ for every $x \in X$ and there exists $\lambda\in \mathbb{R}$ such that $ \lambda |y| \le \|y\| \le (1+\epsilon) \lambda |y|$ for every $y \in Y$. From this it follows that for every $y \in Y$ and every $x \in X$ we have 
    \[\left\langle y,x \right\rangle = \left\langle y,P_Y x \right\rangle  \le |y|\, |P_Y x| \le |y|\, \frac{1}{\lambda}  \| P_Y x\| \le (1+\epsilon) \lambda \frac{|y|^2}{\|y\|} \frac{1}{\lambda} (1+\epsilon) \|x\|  = (1+\epsilon)^2 \frac{\|x\|}{\|y\|}|y|^2, \]
    which implies that every point $y$ in $Y$ is $(2\epsilon + \epsilon ^2)$-good, as claimed.
    
    Conversely, assume that every point in $Y$ is $\epsilon$-good, so that for every $y \in Y$ and every $x \in X$ we have the inequality
    \[\langle y,x \rangle \le (1+\epsilon) \frac{\|x\|}{\| y\|} |y|^2.\]
    Choose $x\in X$. Then $P_Y x\in Y$, so
    \[|P_Yx|^2=\langle P_Yx,P_Yx\rangle=\langle P_Yx,x\rangle\leq (1+\e)\frac{\|x\|}{\|P_Yx\|} |P_Yx|^2,\]
and therefore $\|P_Yx\|\leq(1+\e)\|x\|$. It follows that $Y$ is strongly $(1+\e)$-complemented.   
    
     Now assume for a contradiction that the subspace $Y$ is not strongly $(1+a)$-Euclidean with $0<a$. In particular, this means that we can find two unit vectors $y,w \in Y$ such that $\|y\| = \|w\|(1 +a)$. Without loss of generality we may assume that in the last equality we have $a\le 1/2$, because if we can find unit vectors $y,w \in Y$ such that $\|y\| = \|w\|(1 +a)$ then we can find $y'$ such that $\|y'\| = \|w\|(1 +a')$ for any $a'\le a$. 
     
    Let us consider a sequence of unit vectors $w=x_0, \,x_1 , \ldots , x_{m-1},\, x_m=y$ that are equally spaced along the shortest arc that joins $w$ to $y$ (which is unique, since $w$ cannot equal $-y$). By the pigeonhole principle there exists $i$ such that \[\|x_i\|(1+a)^{1/m} \le \|x_{i+1}\|.\]
         We shall choose $m$ in a way which ensures that $x_i$ is a witness for $x_{i+1}$ not being $\e$-good. Indeed, if we assume that $m$ is at least $3\pi^2/a$ 
         then since the angle between $x_i$ and $x_{i+1}$ is at most $\pi/m$ we get that
\begin{align*}
    \langle x_{i+1},x_i\rangle\frac{\|x_{i+1}\|}{\|x_i\|\,|x_{i+1}|^2}&\geq\cos ( \angle x_i x_{i+1})(1+a)^{1/m}\geq\Bigl(1-\frac{\pi^2}{2m^2}\Bigr)\Bigl(1+\frac am - \frac{a^2}{2m}\Bigr)\\ &\geq1+\frac am-\frac{a^2m^2+\pi^2m+\pi^2a}{2m^3}\geq1+\frac a{2m}.
\end{align*}
    We used the fact that for $0<k<1$ and $a>0$ we have that $(1+a)^k \ge 1+ak-a^2\frac{k(1-k)}{2}\ge  1+ak-\frac{a^2 k}{2}$, and assumptions that $0<a\le 1/2$ and $m\ge 3\pi^2/a$.
    
    It follows that the point $x_{i+1}$ is not $\frac a{2m}$-good. Therefore, if every point is $\e$-good, we must have that $\frac a{\lceil 6\pi^2/a\rceil}\leq\e$, which implies that $a\leq 3\pi\sqrt{\e}$. Thus, we find that $Y$ is strongly $(1+3\pi\sqrt{\e})$-Euclidean, which completes the proof. \end{proof}

In the light of this lemma, we see that Question \ref{strongmilman} is equivalent to the following question.

\begin{question}\label{egoodquestion} Let $\e>0$, $C\geq 1$ and $k\in\N$. Does there exist $n$ such that if $\|\cdot\|$ is a norm on $\R^n$ such that $|x|\leq\|x\|\leq C|x|$ for every $x\in\R^n$, then the space $(\R^n,\|\cdot\|)$ has a subspace $Y$ of dimension $k$ such that every $y\in Y$ is $\e$-good?\end{question}
    
We now reformulate the definition of an $\e$-good point so that it can be applied not just to norms but to more general functions defined on $S^{n-1}$, with the aim of finding a generalization of Question \ref{strongmilman} that does not rely on convexity.

Before we state the result, let us recall that a \emph{support functional} of a norm $\|\cdot\|$ at $x$ is any non-zero linear functional $f$ such that for every $y$ with $\|y\|\le \|x\|$ we have $f(y)\leq f(x)$.
Note that if the norm is differentiable, then writing $f(x)$ for $\|x\|$, we have that any multiple of $f'(x)$ is a support functional at $x$.

\begin{proposition}\label{goodpoints}
Let $(X,\|\cdot\|)$ be a normed space and suppose that $|x|\leq\|x\|\leq C|x|$ for every $x\in X$. For every $\d>0$ there exists $\e>0$ such that if $x\in X$ is any  $\e$-good point, then there exist $y,z$ such that $|x|=|y|$, $|x-y|<\d|x|$, $z$ is a support functional for $y$, and $|y-z|<\d|x|$. Conversely, for every $\e>0$ there exists $\d>0$ such that $x$ is an $\e$-good point if there exist $y,z$ such that $|x-y|<\d|x|$, $z$ is a support functional for $y$, and $|y-z|<\d|x|$.
\end{proposition}

\begin{proof}
We shall do the second part first. Let $0<\e\leq 1$ and suppose that there exist $y,z$ such that $z$ is a support functional for $y$, and $|y-x|$ and $|z-y|$ are both at most $\d |x|$.

Now let $w\in X$. Then
\[\langle w,x\rangle\leq\langle w,y\rangle+\d|w||x|\leq\langle w,z\rangle+2\d|w||x|.\]
But $z$ is a support functional for $y$, so 
\[\langle w,z\rangle\leq\|w\|\,\|z\|^*=\|w\|\frac{\langle y,z\rangle}{\|y\|}\]
We also have that
\[\|y\|\geq\|x\|-C|x-y|\geq\|x\|-C\d|x|\geq(1-C\d)\|x\|.\]
Finally, since $|x-z|\le 2\d |x|$ we have
\[\langle y,z\rangle\leq\langle x,z\rangle+\d|x||z|\leq|x|^2+2\d|x|^2+\d|x|^2(1+2\d)\leq(1+\d)(1+2\d)|x|^2.\]
Putting all this together, we find that
\[\langle w,x\rangle\leq\frac{(1+\d)(1+2\d)}{1-C\d}\frac{\|w\|}{\|x\|}|x|^2+2\d|w||x|\leq\Big(\frac{(1+\d)(1+2\d)}{1-C\d}+2C\d\Big)\frac{\|w\|}{\|x\|}|x|^2.\]
It can be checked that if we set $\d=\e/5C$, then the factor in brackets is at most $1+\e$.

For the other direction, assume that for all $y$ such that $|y|=|x|$ and $|x-y|<\delta |x|$ we have that $|y-z|>\delta|x|$, where $z$ is the support functional at $y$. We will choose $z$ such that $|z|=|y|$.

We can assume that $|x|=1$ and that for every unit vector $y$ with $|y-x|<\d$, we have that $|y-z|\geq\d$. It follows that the component of $z$ orthogonal to $y$ has size at least $\d/2$, which comes from the fact that we chose $z$ with $|z|=|y|=1$ and hence $\sp{z,y}=1-|z-y|^2/2 \le 1-\d^2/2$ and $|z-\sp{y,z}y|^2=1-\sp{y,z}^2$. Further, we have that $|z|\geq 1$ for every $y$ (since $\|y+w\|=\|y\|+\|w\|\geq\|y\|+|w|$ when $w$ is a positive multiple of $y$), so for every unit vector $y$ in the $\delta$-neighborhood of $x$, the component of $z$ orthogonal to $y$ also has size at least $\d/2$.

It follows that for any $\gamma<\delta$  we can find a path on the unit sphere 
that starts at $x$ and ends at a point at distance at least $\g$ from $x$ such that the norm $\|\cdot\|$ decreases at a rate of at least $\d/2$ along the path. This gives us a unit vector $\bar{y}$ such that $|\bar{y}-x|\leq\g$ and \[\|\bar{y}\|\leq\|x\|-\g\d/2\leq\|x\|(1-\g\d/2C).\]

It follows that $\langle x,\bar{y}\rangle> 1-\gamma^2/2$, so
\[\langle x,\bar{y}\rangle>\frac{(1-\gamma^2/2)}{(1-\g\d/2C)}\frac{\|\bar{y}\|}{\|x\|}|x|^2.\]
Setting $\g=\d/2C$, we deduce that $x$ is not $\d^2/8C^2$-good.
\end{proof}

As mentioned before, if $f(x)=\|x\|$ is a differentiable function then $f'(x)$ is a multiple of the support functional at $x\in S^{n}$. A byproduct of the second part of the proof is that we may take $z=\frac{f'(y)}{|f'(y)|}$. Therefore, up to the dependence between $\e$ and $\d$, we have that $x\in S^n$ is $\e$-good if and only if there exists $y\in S^n$ with $|x-y|<\delta$ and $|y-\frac{f'(y)}{|f'(y)|}|<\delta$. 

Note that the latter condition is equivalent to $|f'(y)|(1-\frac{\delta^2}{2})< \sp{y,f'(y)}$ and since we have that $1\le |f'(x)|\le C$, we get
\[   |P_{y^\perp} f'(y)|= |f'(y)-\sp{y,f'(y)}y| \le C(\d^2-\tfrac{\d^4}{4})^{1/2}\le C \d,
\] where $P_{y^\perp}$ denotes the orthogonal projection onto the hyperplane orthogonal to $y$.
Similarly, one can easily find that the condition on $|P_{y^\perp} f'(y)|$ being small implies that $|y-\frac{f'(y)}{|f'(y)|}|$ is small (with some slightly changed constants). Therefore we get that for every $\delta$ there exists $\e$, such that if a point $x\in S^n$ is $\e$-good then there exists $y\in S^n$ with $|x-y|<\delta$ and $|P_{y^\perp} f'(y)|<\delta$.

Since it is enough to look only at unit vectors, if we consider a restriction of $f$, which is the function from the sphere (and call it again $f:S^n\to \R$), then the above condition means that the norm of the gradient of $f$ at $y$ must be small. For simplicity, instead of $\nabla_{S^n}f$, we will write $f'$ for the gradient of a function from the sphere.

This motivates the following definition of a good$^*$ point for any differentiable function: 

\begin{definition}\label{def: good* points}
Let $f:S^n\to \R$ be a differentiable function. We shall say that $x\in S^n$ is \textsl{$\d$-good$^*$} if there exists $y\in S^n$ with $|x-y|<\delta$ and $|f'(y)|<\delta$.
\end{definition}
 Clearly, due to Lemma \ref{goodpoints} and the above remarks, we get that if $f$ is a restriction of a norm, then definitions \ref{def: good* points} and \ref{egooddef} are equivalent (up to the dependence between $\e$ and $\delta$). 

We now note a simple fact about $\delta$-good$^*$ points.

\begin{lemma}\label{bilipschitz}
Let $f:S^n\to\R$ be a differentiable function and let $x\in S^n$ be a $\delta$-good$^*$ point for $f$. Then if $\psi:S^n\to S^n$ is an invertible differentiable function such that both $\psi$ and $\psi^{-1}$ have Lipschitz constant at most $\a$, then $\psi^{-1}(x)$ is an $\a\delta$-good$^*$ point for $f\circ\psi$.
\end{lemma}

\begin{proof}
Assume that $\psi^{-1}(x)$ is not $\alpha \delta$-good$^*$ for $f\circ \psi$. For every $y\in S^n$ such that $d(x,y)<\delta$ we have that $d(\psi^{-1}(x),\psi^{-1}(y))\le \alpha \delta$ and it follows that $ |(f\circ \psi)'(\psi^{-1}(y))|\ge \alpha \delta$. But $(f\circ\psi)'(\psi^{-1}(y))=\psi'(\psi^{-1}(y))^*(f'(y))$, and by the bi-Lipschitz property of $\psi$, this has magnitude at most $\a|f'(y)|$. Hence, we get that $|f'(y)|\ge \delta$, which means that $x$ is not $\delta$-good$^*$ and the result follows.
\end{proof}

We conclude this section with a further question related to Milman's question.

\begin{question}\label{lastversion} Let $\delta>0$ and let $k\in\N$. Does there exist $n$ such that if $f:S^n\to[0,1]$ is any differentiable even function, then there exists a subsphere $S_Y$ of $S^n$ of dimension $k$ that consists only of $\delta$-good$^*$ points?
\end{question}

Note that when analyzing the connection between the above question and Question \ref{strongmilman} one needs to take into account the dependence of $\delta$ and $\epsilon$ explained in Proposition \ref{goodpoints}. Nevertheless, if the answer to the Question \ref{lastversion} is yes (or if it is yes under the additional assumption that $f$ is a Lipschitz function), then also the answer to the ``strong'' Milman's question \ref{strongmilman} is yes.

\subsection{The connection between Milman's question and tennis balls}\label{connection}

Suppose we wish to find a counterexample to Question \ref{strongmilman}. We may as well ask for the norm to be differentiable, and then Proposition \ref{goodpoints} tells us that a point $x$ will be $\e$-good if it is close to a point where the derivative is small. This implies that the set of points with small derivative must have very small measure, since otherwise by measure concentration its $\e/2$-expansion will have measure very close to 1 and a random $k$-dimensional subspace will with high probability have an $\e/2$-net contained in this $\e/2$-expansion, and will therefore live inside the $\e$-expansion.

This kind of observation already rules out many potential methods of constructing counterexamples. For instance, if one chooses a random collection of $N$ unit vectors $u_1,\dots,u_N$ for appropriate $N$ and defines $\|x\|$ to be $\max_i|\langle x,u_i\rangle|$, then for almost all $x\in S^n$ the value of $\|x\|$ will be close to its minimum, which implies that $x$ is $\e$-good. 

However, there do exist norms that are $C$-equivalent to the Euclidean norm and have the property that almost all points are $\e$-bad. A simple example of such a norm is the weighted $\ell_2$-norm given by the formula
\[\|x\|^2=\sum_{i\leq n/2}2x_i^2+\sum_{i>n/2}x_i^2.\]
Letting $A$ be the diagonal matrix with the first $n/2$ entries equal to $2$ and the rest equal to $1$, we can write the right-hand side as $\langle x,Ax\rangle$. When $x\ne 0$, the derivative of this norm at $x$ is $Ax/\|x\|$, or if we regard the norm as a function defined on $S^{n-1}$, it is the projection of $Ax/\|x\|$ on to the subspace orthogonal to $x$. Thus, a point $x$ is 0-good if and only if $x$ is an eigenvector of $A$, which is the case if and only if it belongs to one of the two eigenspaces $\langle e_1,\dots,e_{n/2}\rangle$ or $\langle e_{n/2+1},\dots,e_n\rangle$. It is a straightforward exercise to prove the more precise result that for every $\eta>0$ there exists $\e>0$ such that $x$ is $\e$-good only if the distance from $x$ to one of these two subspaces is at most $\eta$. Since the set of such $x$ has exponentially small measure, we have an example of a norm where it is not the case that almost all points are $\e$-good. 

From the perspective of Milman's question, this may seem a dubious example, since the space is isometric to a Euclidean space, and therefore as far from a counterexample as it is possible to be. But as we shall see later there are examples to which this criticism does not apply. 

However, the main point we wish to make here is that even this example gives us a strategy for building a counterexample to Question \ref{lastversion}. Let $S_1$ and $S_2$ be the unit spheres of the two eigenspaces above, and suppose that we can find a function $\psi$ satisfying the conditions of Lemma \ref{bilipschitz} such that $\psi(S_1)$ and $\psi(S_2)$ are both tennis balls. By Lemma \ref{bilipschitz} there will exist $\e>0$ independent of $n$ such that every point that is at least $\e$ away from $\psi(S_1)\cup\psi(S_2)$ is not $\e$-good$^*$. Since $\psi(S_1)$ and $\psi(S_2)$ are tennis balls, there is also $\eta>0$ such that the $\eta$-expansions of $\psi(S_1)$ and $\psi(S_2)$ do not contain the spheres of any 2-dimensional subspaces. Also, if any such sphere contains a point close to $\psi(S_1)$ and a point close to $\psi(S_2)$, then since those sets are far apart (by the bi-Lipschitz property of $\psi$), it must also contain points far from both sets. Therefore, the function given by the composition $x\mapsto \|\psi^{-1}(x)\|$ is a counterexample.

Our main theorem (Theorem \ref{main}) does something very similar to this. However, since the map $\psi$ provided by the theorem creates $\lfloor\a n\rfloor$-dimensional tennis balls with $\a$ significantly smaller than 1/2, we replace the diagonal map $A$ above by a map that has roughly $\a^{-1}$ eigenspaces of dimension roughly $\a n$ with well-separated eigenvalues. In that way, we do indeed obtain a counterexample to Question \ref{lastversion}.  
\fi

\section{The construction}

Throughout this note we shall write $\|\cdot\|$ for the norm given by the formula 
\begin{align}\label{L2 norm}
    \|x\|^2=n^{-1}\sum_{i=1}^nx_i^2.
\end{align}
The advantage of the factor $n^{-1}$ on the right-hand side is that a typical coordinate of a random vector of norm 1 has order of magnitude 1 rather than order of magnitude $n^{-1/2}$. This norm is often called the $L_2^n$ norm on $\R^n$, and we write $L_2^n=(\R^n, \|\cdot\|)$.  It is the Euclidean norm most commonly used in additive combinatorics. Following the standard terminology in that field, we shall sometimes write the right-hand side of the formula above as~$\E_ix_i^2$. Moreover, from this point on we shall think of spheres concretely,  so $S^{n-1}$ will denote the unit sphere of $L_2^{n}$.

\subsection{The tennis ball map}\label{tennisballmap}

We are aiming to prove Theorem \ref{main}, or in other words to prove that there exists a continuous map (in fact it will be bi-Lipschitz) $\psi:S^{n-1}\to S^{n-1}$ that preserves antipodal points, with the property that if $X$ is a random $\lfloor\a n\rfloor$-dimensional subsphere of $S^{n-1}$, then with high probability $\psi(X)_\e$ contains no linear subsphere of dimension 1. We shall achieve this by identifying a set $\Gamma\subset S^{n-1}$ such that with high probability $\psi(X\cap S^{n-1})\subset\Gamma$, or equivalently $X\cap S^{n-1}\subset\psi^{-1}(\Gamma)$, and such that every great circle contains a point that does not belong to $\Gamma_\e$.

These properties are clearly in tension with each other: we need $\Gamma$ to have small measure, or else its expansion $\Gamma_\e$ will contain a great circle, but on the other hand we also need $\psi^{-1}(\Gamma)$ to have measure very close to 1, or else it will not contain almost all $\lfloor\a n\rfloor$-dimensional linear subspheres.

In order to resolve this tension, we define a map that takes ``typical'' vectors to highly ``atypical'' vectors. Let $k$ be a large positive integer to be chosen later, let $\lambda>1$, and define $s=\lambda^{1/2k}$. (The parameter $\lambda$ will later be chosen to be 4, but we write most of the arguments in slightly greater generality in order to emphasize a certain flexibility in our construction and to make the role of this parameter more explicit.)

\begin{definition}
Let $\v:\R\to\R$ be a strictly increasing continuous odd function such that for every integer $m$ we have
\[\v(s^{2mk-1})=s^{2mk+1}  \ \text{ and } \  \v(s^{2mk+1})=s^{2(m+1)k-1} \]
and such that $\varphi$ is linear on $(s^{2mk-1},s^{2mk+1})$ and $(s^{2mk+1},s^{2(m+1)k-1})$ for every $m$.


When $x=(x_1,\dots,x_n)$ is a vector in $\R^n$ we shall abuse notation by writing $\v(x)$ to denote the vector $(\v(x_1),\dots,\v(x_n))$. The \emph{tennis ball map} is a function $\psi:S^{n-1}\to S^{n-1}$,  which is a normalized version of $\v$, given by the formula
\[\psi(x)=\frac{\v(x)}{\|\v(x)\|}.\]
\end{definition}

Note that the graph of $\v$ has a kind of staircase shape with steps of sizes that grow exponentially (see Figure \ref{fig: phi}). Indeed, as $x$ increases from $s^{2mk-1}$ to $s^{2mk+1}$, which is only a small change proportionately speaking, $\v(x)$ increases from $s^{2mk+1}$ to $s^{2(m+1)k-1}$, which is an increase by a factor of almost $\lambda$. Similarly, as $x$ increases from $s^{2mk+1}$ to $s^{2(m+1)k-1}$, which is about $\lambda$ times as big, $\v(x)$ increases from $s^{2(m+1)k-1}$ to $s^{2(m+1)k+1}$, which is only a small increase.

\begin{figure}
\begin{center}
\begin{tikzpicture}[yscale=0.5, xscale=0.8,cap=round,
tangent/.style={%
in angle={(180+#1)} ,
Hobby finish ,
designated Hobby path=next , out angle=#1,
},>=stealth]

\draw[->,thick] (-6,0)--(6,0) node[right]{$x$};
\draw[->,thick] (0,-5)--(0,5) node[above]{$y$};

\draw (0,0)--(0.01719,0.0568);
\draw (0.01719,0.0568)--(0.0568,0.06875);
\draw (0.0568,0.06875)--(0.06875,0.227);
\draw (0.06875,0.227)--(0.227,0.275);
\draw (0.227,0.275)--(0.275,0.91);
\draw (0.275,0.91)--(0.91, 1.1);
\draw (0.91, 1.1)--(1.1,3.64);
\draw (1.1,3.64)--(3.64,4.4);
\draw (3.64,4.4)--(3.7,5);

\draw [ultra thick, blue] (4.4,0)--(5.9,0);
\draw [ultra thick, blue] (3.64,0)--node [midway,below] { $A$}(1.1,0);
\draw [ultra thick, blue] (0.275,0)--(0.91,0);
\draw [ultra thick, blue] (0.06875,0)--(0.227,0);

\draw [ultra thick, blue] (-4.4,0)--(-5.9,0);
\draw [ultra thick, blue] (-3.64,0)--(-1.1,0);
\draw [ultra thick, blue] (-0.275,0)--(-0.91,0);
\draw [ultra thick, blue] (-0.06875,0)--(-0.227,0);

\draw[dotted] 
    (0,4.4)  -- (3.64,4.4) -- (3.64,0) node[below]{$ $}
    (0,3.64)  -- (1.1,3.64) -- (1.1,0) node[below]{$ $}  
    (0,1.1)  -- (0.91, 1.1) -- (0.91,0) node[below]{$ $}
    (0,0.91) -- (0.275,0.91) -- (0.275,0) 
    (0,0.275)  -- (0.227,0.275) -- (0.227,0) 
    (4.4,0) -- (4.4,5) ;

\draw (0,0)--(-0.01719,-0.0568);
\draw (-0.01719,-0.0568)--(-0.0568,-0.06875);
\draw (-0.0568,-0.06875)--(-0.06875,-0.227);
\draw (-0.06875,-0.227)--(-0.227,-0.275);
\draw (-0.227,-0.275)--(-0.275,-0.91);
\draw (-0.275,-0.91)--(-0.91, -1.1);
\draw (-0.91, -1.1)--(-1.1,-3.64);
\draw (-1.1,-3.64)--(-3.64,-4.4);
\draw (-3.64,-4.4)--(-3.7, -5);

\draw [ultra thick,red] (0,3.64)--node [midway,left] { $B$}(0,4.4);
\draw [ultra thick,red] (0,0.91)--(0,1.1);
\draw [ultra thick,red] (0,0.227)--(0,0.275);
\draw [ultra thick,red] (0,-3.64)--(0,-4.4);
\draw [ultra thick,red] (0,-0.91)--(0,-1.1);
\draw [ultra thick,red] (0,-0.227)--(0,-0.275);

\draw[dotted] 
    (0,-4.4)  -- (-3.64,-4.4) -- (-3.64,0) 
    (0,-3.64)  -- (-1.1,-3.64) -- (-1.1,0)   
    (0,-1.1)  -- (-0.91,- 1.1) -- (-0.91,0) 
    (0,-0.91) -- (-0.275,-0.91) -- (-0.275,0)   
    (0,-0.275)  -- (-0.227,-0.275) -- (-0.227,0)
    (-4.4,0) -- (-4.4,-5) ;

\end{tikzpicture}\caption{The ``staircase function" $\v$. } \label{fig: phi}
\end{center} \end{figure}
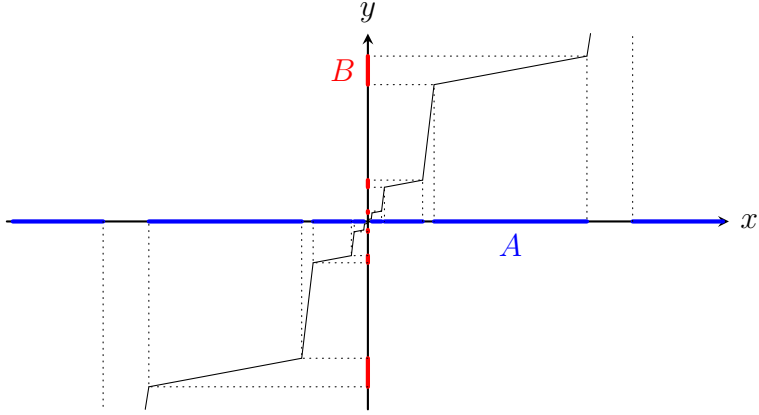

In particular, writing $A_m$ for the ``wide'' interval and $B_m$ for the ``narrow'' interval, i.e.
\begin{align}\label{sets Am and Bm}
    A_m=[s^{2mk+1},s^{2(m+1)k-1}] \ \ \ \text{ and } \ \ \  B_m=[s^{2mk-1},s^{2mk+1}]
\end{align}
 we have that \[\v(A_m)=B_{m+1}\] for every $m$. Moreover, if by $A$ we denote the union $\bigcup_m(A_m\cup(-A_m))$ and by $B$ the union $\bigcup_m(B_m\cup(-B_m))$, then we also have that $\v(A)=B$.

Note that if $x,y\in B$, then $xy^{-1}$ belongs to an interval of the form $[s^{2mk-2},s^{2mk+2}]$, or minus such an interval. In other words, we ensure that $B$ is a ``geometric-progression-like'' set in the sense that $BB^{-1}$ is not much larger than $B$ itself. It follows that $B^2B^{-2}$ consists of all points that belong to an interval of the form $[s^{2mk-4},s^{2mk+4}]$. This fact will be useful later on. Another observation we shall use is that 
\begin{align}\label{phi is lipschitz}
    \|x\|\leq\|\v(x)\|\leq \lambda\|x\|
\end{align} for every $x$. This implies that $\psi$ is a Lipschitz function. 

 Similarly,  we observe that the inverse of $\varphi$ satisfies $\frac 1\lambda \|x\| \le\|\v^{-1}(x)\| \le \|x\|$, and hence we get that $\psi^{-1}$ is Lipschitz as well. Thus, the tennis ball map $\psi$ is indeed bi-Lipschitz as claimed.

In what follows we shall show that the tennis ball map $\psi$ takes random linear subspheres of appropriate dimension to tennis balls.

\subsection{The definition of the set $\Gamma$}

Now that we have defined the tennis ball map, let us discuss the set $\Gamma$. The rough idea is that $\Gamma$ is the set of points $x\in S^{n-1}$ with almost all their coordinates in $B$. Since $A$ has a small complement, one would expect almost all coordinates of a random vector to belong to $A$, and indeed this is the case. It forms the basis of a probabilistic argument that shows that with high probability every $x\in S^{n-1} \cap X$ has the property that almost every coordinate of $x$ belongs to $A$, which implies that almost every coordinate of $\v(x)$ belongs to $B$. Thus a ``typical'' vector (one with almost all coordinates in the large set $A$) is mapped to a highly ``atypical'' vector (one with almost all coordinates in the small set $B$). 

This is a slight oversimplification, because of the normalization that replaces $\v$ by $\psi$. The actual definition of $\Gamma$ concerns the \emph{ratios} of the coordinates rather than their actual values.

Now let us give some more details. For the purposes of this problem, it is more natural, when talking about a unit vector $x$, to attach a weight of $x_i^2$ to the $i$th coordinate. For example, the statement ``almost every ratio $x_ix_j^{-1}$ belongs to $BB^{-1}$'' should be interpreted as meaning that
\begin{align*}
    \sum_{x_ix_j^{-1}\in BB^{-1}} x_i^2x_j^2 \geq(1-\e)\sum_{i,j=1}^n x_i^2x_j^2
\end{align*}
for some small $\e$, and similarly for other statements about coordinates. More precisely, with every vector $x$ we define an associated probability measure as follows.

\begin{definition}
For any $x\in \R^n\setminus \{0\}$ define a measure $\mu_x$ on $\{1,2,\dots,n\}$ by the formula \[\mu_x(J)=\frac{\|P_Jx\|^2}{\|x\|^2},\] where $P_J$ is the coordinate projection to the set $J$. 
\end{definition}

Let us now fix some useful notation.  For functions $f: \{1,2,\ldots,n\} \to \R$  and $g: \{1,2,\ldots,n\}\times \{1,2,\ldots,n\} \to \R$ we write 
\[\E_i^xf(i) = \|x\|^{-2}\E_ix_i^2f(i) \ \ \ \ \text{ and } \ \ \ \ \E_{i,j}^x g(i,j)= \|x\|^{-4}\E_{i,j}x_i^2x_j^2g(i,j),\]
where, as mentioned earlier, $\E_i$ and $\E_{i,j}$ are the obvious averages. 

Then, given a property $Q$ of integers in the set $\{1,2,\dots,n\}$ we define
\[ \P^x_i[Q(i)]  = \E_i^x \mathbbm{1}_Q(i)= \mu_x\{i:Q(i)\}.\]
Similarly, we shall write 
\[\P^x_{i,j}[Q(i,j)]=\E_{i,j}^x \mathbbm{1}_Q(i,j) = (\mu_x\times\mu_x)\{(i,j):Q(i,j)\},\] 
where $\mu_x\times\mu_x$ is the product measure.

We are now ready to define the set $\Gamma$.
\begin{definition}
Let $\Gamma \subset S^{n-1}$ be the set given by
\begin{align}\label{def: gamma set}
    \Gamma =\{ y\in S^{n-1}: \ \P_{i,j}^y [y_i/y_j\in BB^{-1}] \ge 1-2\lambda^2\be\},
\end{align}
for a suitable choice of the parameter $\be$.
\end{definition}

Also important to us will be a set $\Delta$ defined by
\begin{align}\label{def: delta set}
    \Delta = \{x\in S^{n-1}: \ \P^x_i[x_i\in A]<1-\be\}.
\end{align}
In this section we will be more interested in the complement \[ \Delta^c = \{x\in S^{n-1}: \ \P^x_i[x_i\in A]\ge1-\be\}\] but our choice of notation is more convenient in further sections.

\begin{proposition}
The image of $ \D^c$ under the tennis ball map $\psi$ is a subset of $\Gamma$.
\end{proposition}

\begin{proof}
Let $x\in \D^c$ and note that if $\P^x_i[x_i\in A]\geq 1-\be$, then $\P^x_i[\v(x_i)\in B]\geq 1-\be$, by the definition of the function $\v$. Moreover, from \eqref{phi is lipschitz} it follows that $\|\v(x)\|^2$ always lies between $\|x\|^2$ and $\lambda^2 \|x\|^2$ and therefore $\mu_{\v(x)}(E)\leq\lambda^2\mu_x(E)$ for every vector $x\in\R^n$ and every set $E\subset\{1,2,\dots,n\}$. From this it follows (considering complements) that
\[\P_i^{\v(x)} [\v (x_i) \in B] >1-\lambda^2\be.\]
Since this is true for every $x\in \Delta^c$ we get that 
\begin{align}\label{eq: phi delta subset}
    \v(\Delta^c)\subset\{y:\P^y_i[y_i\in B]>1-\lambda^2\be\}.
\end{align}
In order to show the inclusion for $\psi$, note that if we have $y$ as above then 
\[\P^{y}_{i,j}[y_i\in B\ \hbox{and}\ y_j\in B]\geq 1-2\lambda^2\be,\]
which in turn gives us
\[\P^{y}_{i,j}[y_i/y_j\in BB^{-1}]\geq 1-2\lambda^2\be.\]
From this follows that such a $y$ belongs to $\Gamma$. Since $\Gamma$ is invariant under positive scalar multiples, we conclude using \eqref{eq: phi delta subset} that  $\psi(\Delta^c)\subset\Gamma$.
\end{proof}

In the next section we shall prove that the expansion $\Gamma_\e$ contains no great circle, and in the following one we shall prove that for a suitable constant $\a>0$, a random linear subsphere of dimension $\lfloor \a n\rfloor$ is contained in $\Delta^c$ with high probability.

\section{Proving that no great circle is contained in $\Gamma_\e$}\label{sec2.3}

We begin with a couple of lemmas that help us to describe the set $\Gamma_\e$. Recall that if $x$ is a vector in $\R^n$, then $\|x\|$ is its $L_2$ norm (as defined in \eqref{L2 norm}). 

\begin{lemma}\label{switchprob}
Let $y,z$ be unit vectors in $L_2^n$ with $\|y-z\|\leq\e$, and let $E$ be a subset of $\{1,2,\dots,n\}$. Then $\bigl|\P^y_i[E]-\P^z_i[E]\bigr|\leq 2\e$.
\end{lemma}

\begin{proof}
The left-hand side is equal to  $|\E_i(y_i^2-z_i^2)\mathbbm 1_E(i)|\leq\E_i|y_i^2-z_i^2|$. 
By the Cauchy-Schwarz inequality it follows that
\begin{align*}\E_i|y_i^2-z_i^2|&=\E_i|y_i-z_i|\,|y_i+z_i|\\
&\leq \|y-z\|\|y+z\|\\
&\leq 2\e,
\end{align*}
which proves the result.
\end{proof}

Recall from \S\ref{tennisballmap} that $BB^{-1}$ is the union of all intervals of the form $[s^{2mk-2},s^{2mk+2}]$, and $B^2B^{-2}$ is the union of all intervals of the form $[s^{2mk-4},s^{2mk+4}]$, where $s=\lambda^{1/2k}$ was one of the parameters used to define the ``staircase function'' $\v$. It follows that if $t\in BB^{-1}$ and $u\notin B^2B^{-2}$, then  $|t/u|\notin[s^{-2},s^2]$, which implies in particular that $|t/u-1|\geq 1-s^{-2}$.

\begin{lemma}\label{Gammaeps}
Let $\tau=1-s^{-2}$ and $\e< \tau^2$. Then for $z\in\Gamma_\e$ we have
\[\P^z_{i,j}[z_i/z_j\in B^2B^{-2}]\geq 1-2\lambda^2\be-6\e.\]
\end{lemma}

\begin{proof}
Let $y\in\Gamma$ with $\|y\|=1$ be such that $\|y-z\|\leq\e$. Then $\P^y_{i,j}[y_i/y_j\in BB^{-1}]\geq 1-2\lambda^2\be$, or equivalently
\[\E_{i,j} y_i^2y_j^2\mathbbm{1}_{[y_i/y_j\in BB^{-1}]}\geq 1-2\lambda^2\be.\]
By Lemma \ref{switchprob} (and recalling that $\E_i y_i^2=1$), it follows that 
\[\E_{i,j} (y_i^2-z_i^2)y_j^2\mathbbm{1}_{[y_i/y_j\in BB^{-1}]} \le 2\e\]
and hence combining both inequalities gives
\begin{align}\label{eq: first}
    \E_{i,j} z_i^2y_j^2\mathbbm{1}_{[y_i/y_j\in BB^{-1}]}\geq 1-2\lambda^2\be-2\e.
\end{align}

If the conclusion that $\P^z_{i,j}[z_i/z_j\in B^2B^{-2}]\geq 1-2\lambda^2\be-6\e$ is not true, then we would have
\[\E_{i,j} z_i^2z_j^2\mathbbm{1}_{[z_i/z_j\notin B^2B^{-2}]}\geq 2\lambda^2\be+6\e,\]
and by Lemma \ref{switchprob} again it follows that
\begin{align}\label{eq: second}
    \E_{i,j} z_i^2y_j^2\mathbbm{1}_{[z_i/z_j\notin B^2B^{-2}]}\geq 2\lambda^2\be+4\e.
\end{align}
By summing \eqref{eq: first} and \eqref{eq: second}, and estimating trivially the probability of the union of the events by 1, we deduce that
\[\E_{i,j} z_i^2y_j^2\mathbbm{1}_{[y_i/y_j\in BB^{-1}\ \hbox{and}\ z_i/z_j\notin B^2B^{-2}]}\geq 2\e.\]

As remarked before the lemma, if $y_i/y_j\in BB^{-1}$ and $z_i/z_j\notin B^2B^{-2}$, then $|\frac{y_iz_j}{y_jz_i}-1|> \tau$. It follows that
\[\E_{i,j} z_i^2y_j^2\Bigl(\frac{y_iz_j}{y_jz_i}-1\Bigr)^2>2\tau^2\e.\]
But since $\E_i y_i^2=\E_i z_i^2=1$ we have
\begin{align*}\E_{i,j} z_i^2y_j^2\Bigl(\frac{y_iz_j}{y_jz_i}-1\Bigr)^2&=\E_{i,j}(y_iz_j-z_iy_j)^2\\
&=\E_{i,j} (y_i^2z_j^2+y_j^2z_i^2-2y_iy_jz_iz_j)\\
&=2-2\langle y,z\rangle^2.
\end{align*}
Furthermore, if $2-2\langle y,z\rangle^2>2\tau^2\e$, then $\langle y,z\rangle^2< 1-\tau^2\e$, which implies that $\langle y,z\rangle<1-\tau^2\e/2$, which in turn means that
\[\|y-z\|^2=2-2\langle y,z\rangle> \tau^2\e,\]
and therefore that $\|y-z\|>\tau\sqrt\e$. However, since by assumption we have $\tau>\sqrt\e$, this is a contradiction.
\end{proof}

\begin{corollary}
It follows directly from Lemma \ref{Gammaeps} that \[\Gamma_\e \subset \{z:\P^z_{i,j}[z_i/z_j\in B^2B^{-2}]\geq 1-2\lambda^2\be-6\e\}.\]
\end{corollary}

This bigger set resembles $\Gamma$ but is defined using slightly different parameters. We now turn to the proof that every great circle contains a point that does not belong to this slightly expanded $\Gamma$-like set.

\subsection{Finding a suitable point in an arbitrary 2-dimensional subspace}
Let $Y$ be a 2-dimensional subspace of $L_2^n$ and let $\{u,v\}$ be an orthonormal basis for $Y$. Then the unit sphere of $Y$ consists of vectors $u\cos\theta+v\sin\theta$. The $i$th coordinate of such a vector, $u_i\cos\theta+v_i\sin\theta$, can be rewritten as $a_i\sin(\theta+\phi_i)$, where $a_i=\sqrt{u_i^2+v_i^2}$ and $\phi_i$ is chosen such that $a_i\sin\phi_i=u_i$ and $a_i\cos\phi_i=v_i$. Let $a=(a_1, \ldots ,a_n)$ and note that $\|a\|^2=2$.

We start by proving that there are plenty of pairs $(i,j)$ such that $\phi_i$ is not close to $\phi_j$ or $-\phi_j$. 

\begin{lemma} \label{separatedangles}
With $a_1,\dots,a_n$ and $\phi_1,\dots,\phi_n$ as above, we have the inequality
\[\P^a_{i,j}[\cos(2(\phi_i-\phi_j))\leq 1/2]\geq 1/3.\]
\end{lemma}

\begin{proof}
Since $u,v$ are fixed unit vectors (in $L_2^n$) we have that $\|a\|^2=\E_i (u_i^2+v_i^2)=2$ and hence $\E^a_i\sin^2(\theta+\phi_i)=\|a\|^{-2}\E_i a_i^2 \sin^2(\theta +\phi_i)=\frac{1}{2}$  for every $\theta$. 
Therefore, we find on differentiating with respect to $\theta$ that
\[2\mathbb E^a_i\sin(\theta+\phi_i)\cos(\theta+\phi_i)=\mathbb E^a_i\sin(2\theta+2\phi_i)=0\]
for every $\theta$, and hence, on differentiating again, that
\[\mathbb E^a_i\cos(2\theta+2\phi_i)=0\]
for every $\theta$ as well.

From that it follows that
\[\mathbb E^a_{i,j}\Bigl(\cos(2\theta+2\phi_i)\cos(2\theta+2\phi_j)+\sin(2\theta+2\phi_i)\sin(2\theta+2\phi_j)\Bigr)=\mathbb E^a_{i,j}\cos(2(\phi_i-\phi_j))=0.\]
Let $F$ be the event that $\cos(2(\phi_i-\phi_j))\leq 1/2$. We have seen that $\E^a_{i,j}\cos(2(\phi_i-\phi_j))=0$, and we also know that $\cos(2(\phi_i-\phi_j))\in[-1,1]$. So, using the total probability formula we get
\[0=\mathbb E^a_{i,j}\cos(2(\phi_i-\phi_j))\geq\frac 12\P^a_{i,j}[F^c]-\P^a_{i,j}[F]=\frac 12-\frac 32\P^a_{i,j}[F],\]
from which the desired inequality follows.
\end{proof}

Next, we need a technical lemma that will help us to show that if $\phi_i$ is not approximately $\pm\phi_j$, then $\sin(\theta+\phi_i)/\sin(\theta+\phi_j)$ is not often close to an element of some given geometric progression.

\begin{lemma} \label{cot}
Let $\theta$ be chosen randomly from $[-\pi,\pi]$ and let $0<a<b$. Then 
\[ \P \left[a \le \cot\theta\le b \right] \le \frac{b-a}{\pi(1+a^2)},\] and the same bound holds for the probability that $\cot\theta\in[-b,-a]$.
\end{lemma}

\begin{proof}
Since $\cot$ is periodic with period $\pi$ and is decreasing in the interval $(0,\pi)$, the probability in question is $(\cot^{-1}a-\cot^{-1}b)/\pi$. By the mean value theorem, $\cot^{-1}a-\cot^{-1}b$ is at most $|a-b|$ times the absolute value of the derivative of $\cot^{-1}$ at $a$. Since that derivative is $-1/(1+a^2)$, the first result follows. The second then holds by symmetry.
\end{proof}

Recall once again that $B^2B^{-2}$ is the set of all real numbers $x$ such that \[|x|\in[\lambda^ms^{-4},\lambda^ms^{4}]\] for some positive integer $m$.

The main point about the bound in the next lemma is not its exact form, but simply that it is $O(\xi)$ except when $\phi_i$ is close to $\phi_j$ or $\phi_j+\pi$.

\begin{lemma} \label{badratio}
Let $\xi=s^{8}-1$ and let $\theta\in[0,2\pi]$ be chosen uniformly at random. Then  
\[\P \left[ \frac{a_i\sin(\theta+\phi_i)}{a_j\sin(\theta+\phi_j)}\in B^2B^{-2}\right] \le \frac{4 \xi\lambda}{\pi(\lambda-1)}(4+|\cot(\phi_i-\phi_j)|).\]
\end{lemma}

\begin{proof}
The distribution of $\frac{a_i\sin(\theta+\phi_i)}{a_j\sin(\theta+\phi_j)}$ is the same as the distribution of 
\[\frac{a_i\sin(\theta+\phi_i-\phi_j)}{a_j\sin(\theta)}=\frac{a_i}{a_j}\Bigl(\cos(\phi_i-\phi_j)+\sin(\phi_i-\phi_j)\cot\theta\Bigr).\]
Therefore, we are interested in the probability that $\cot\theta\in\frac{a_j}{a_i\sin(\phi_i-\phi_j)}B^2B^{-2}-\cot(\phi_i-\phi_j)$.

Let $t=|\frac{a_j s^{-4}}{a_i\sin(\phi_i-\phi_j)}|$.
Then \[\frac{a_j}{a_i\sin(\phi_i-\phi_j)}B^2B^{-2}= \bigcup_{m} \Big([t\lambda^m,t\lambda^ms^{8}]\cup[-t\lambda^m s^{8},-t\lambda^m]\Big).\]
 By Lemma \ref{cot}, we get the bound  \[ \P \left[\cot\theta\in\Bigl[t\lambda^m-\cot(\phi_i-\phi_j),t\lambda^ms^{8}-\cot(\phi_i-\phi_j)\Bigr] \right] \le \frac{\xi t\lambda^m}{\pi}\] for all $m$, and in addition if $t\lambda^m \ge 2\cot(\phi_i-\phi_j)$, then since $s^{-2}<1$ we have an upper bound of \[\frac{\xi t\lambda^m}{\pi(1+t^2\lambda^{2m}/4)}\leq \frac{4\xi}{\pi t\lambda^m}.\]

If $\cot(\phi_i-\phi_j)\geq 0$, then the probability that $\cot\theta+\cot(\phi_i-\phi_j)$ lies in the positive part of $\frac{a_j}{a_i\sin(\phi_i-\phi_j)}B^2B^{-2}$ is therefore at most $\xi/\pi$ multiplied by the sum
\[\sum_{t\lambda^m\leq S}t\lambda^m+4\sum_{t\lambda^m>S}\frac 1{t\lambda^m},\]
where $S=\max\{2\cot(\phi_i-\phi_j),1\}$. By the formula for the sum of a geometric progression, and recalling that $\lambda \ge \frac{3}{2}$, the first sum can be estimated by
\[\sum_{t\lambda^m\leq S}t\lambda^m = t \sum_{m=-\infty}^{m_0} \left(\frac{1}{\lambda}\right)^m = t \frac{(1/\lambda)^{-m_0}}{1-\lambda^{-1}} \le S \frac{\lambda}{\lambda-1},\]
and similarly the second sum is at most $S^{-1}\frac{\lambda}{\lambda-1}$. Therefore, the total is at most \[\frac{(S+4S^{-1})\lambda}{\lambda-1}\leq \frac{(5+2\cot(\phi_i-\phi_j))\lambda}{\lambda-1}.\]
Therefore, we obtain an answer of at most $\xi\lambda(5+2\cot(\phi_i-\phi_j))/\pi(\lambda-1)$.

If $\cot(\phi_i-\phi_j)<0$, then in the same way we get that the probability that $\cot\theta\in\bigl[t\lambda^m-\cot(\phi_i-\phi_j),t\lambda^ms^{8}-\cot(\phi_i-\phi_j)\bigr]$ is at most $\xi t\lambda^m/\pi$ for all $m$ but now it is also at most $\xi/\pi t\lambda^m$ for all $m$. Using the first bound when $t\lambda^m\leq 1$ and the second when $t\lambda^m>1$, we obtain an upper bound of at most \[\frac{2\xi \lambda}{\pi(\lambda-1)}.\]

Considering the negative part of $B$ as well and combining these two estimates, we obtain the result stated.
\end{proof}

Let us recall that in Subsection \ref{tennisballmap} we defined $s=\lambda ^{1/2k}$, where $\lambda>1$ and $k\in \N$ large enough.  We have further, for convenience, defined parameters  $\tau=1-s^{-2}$ and $\xi=s^8-1$. 

\begin{corollary} \label{2Dsubspaces}
If $\tau\le10^{-4}$ and $\lambda\geq 3/2$, then in every 2-dimensional subspace of $L_2^n$ there is a vector $y$ such that 
\[\P^y_{i,j}[y_i/y_j\notin B^2B^{-2}]\geq 1/8.\]
\end{corollary} 

\begin{proof}
Let a typical unit vector $y$ in the subspace have $i$th coordinate $y_i=a_i\sin{(\theta+\phi_i)}$. We will bound the desired probability from below by adding an additional constraint. We will consider the probability that $\frac{y_i}{y_j}\notin B^2B^{-2} \text{ and } |\cot(\phi_i-\phi_j)|\le 2$, which can be found by calculating the expected probability that $|\cot(\phi_i-\phi_j)|\le 2$ and subtracting from it the probability of the event $\{\frac{y_i}{y_j}\in B^2B^{-2} \text{ and } |\cot(\phi_i-\phi_j)|\le 2\}$. 

We have for each $i,j$ that
\begin{align*}
\E_\theta\sin^2(\theta+\phi_i)\sin^2(\theta+\phi_j)&=\frac 14\E_\theta\Bigl(\cos(\phi_i-\phi_j)-\cos(2\theta+\phi_i+\phi_j)\Bigr)^2\\
&=\frac 14\Bigl(\cos^2(\phi_i-\phi_j)+\E_\theta\cos^2(2\theta+\phi_i+\phi_j)\Bigr)\\
&=\frac 14\Bigl(\cos^2(\phi_i-\phi_j)+\frac 12\Bigr)\geq \frac 18.
\end{align*}
Here $\E_\theta$ is just the usual average over $\theta\in[0,2\pi)$. Recall that $y$ is such that $\|y\|=1$ and $y_i=a_i \sin(\theta +\phi_i)$ for some such $\theta$. For any event $Q$ that depends on two coordinates $i,j$, we get
\begin{align*}\E_y\P^y_{i,j}[Q]
&=\E_\theta\E_{i,j}a_i^2a_j^2\sin^2(\theta+\phi_i)\sin^2(\theta+\phi_j)\mathbbm 1_Q(i,j)\\
&\ge\frac 18 \,\E_{i,j}a_i^2a_j^2\mathbbm 1_Q(i,j)\\
&=\frac{\|a\|^4}{8}\E ^a_{i,j} \mathbbm 1_Q(i,j)\\
&= \frac{1}{2}\P^a_{i,j}[Q],
\end{align*}
where we used that $\|a\|^2=2$.

It is easy to check the identity  $\cot^2 \a=\frac{2\cos^2\a}{1-\cos(2\a)}$, so if $\cos(2(\phi_i-\phi_j))\leq 1/2$, then $|\cot(\phi_i-\phi_j)|\leq 2$. Therefore,
\begin{align*}\E_y\P^y_{i,j}[|\cot(\phi_i-\phi_j)|\leq 2]&\geq\E_y\P^y_{i,j}[\cos(2(\phi_i-\phi_j))\leq 1/2]\\
&\ge \frac{1}{2}\,\P^a_{i,j}[\cos(2(\phi_i-\phi_j))\leq 1/2] \,
\ge \,  \frac 16,
\end{align*}
where the last inequality follows from Lemma \ref{separatedangles}.

Now, by Lemma \ref{badratio}, if $y$ is a random such vector, then for each $i,j$ the probability that $y_i/y_j\in B^2B^{-2}$ is at most $\frac{4\xi\lambda}{\pi(\lambda-1)}\Bigl( 4+ |\cot(\phi_i-\phi_j)| \Bigr) \leq 4\xi\Bigl( 4+ |\cot(\phi_i-\phi_j)| \Bigr)$, where the last inequality uses the fact that $\lambda \ge 3/2$, which implies that $\lambda/\pi(\lambda-1)\leq 1$.  Note also that since $y_i^2\leq a_i^2$ for each $i$, and $\E_iy_i^2=\frac 12\E_ia_i^2$, we have that $\P^y_{i,j}[Q(i,j)]\leq 4\P_{i,j}^a[Q(i,j)]$ for every event $Q(i,j)$ that depends on two coordinates $i,j$. It follows that
\begin{align*}\E_y\P^y_{i,j} [|\cot(\phi_i-\phi_j)|\leq 2\ & \hbox{and}\ y_i/y_j\in B^2B^{-2}] \\
&\leq 4\E_y\P^a_{i,j}[|\cot(\phi_i-\phi_j)|\leq 2\ \hbox{and}\ y_i/y_j\in B^2B^{-2}]\\
&\leq 16\xi(4+2)\\
&=96\xi.
\end{align*}

 Together with the estimate in the previous paragraph, this implies that
\[\P^y_{i,j}\bigl[|\cot(\phi_i-\phi_j)|\leq 2\ \hbox{and}\ y_i/y_j\notin B\bigr]\geq 1/6-96\xi.\] 

It is straightforward to check that our assumption that $\tau\leq 10^{-4}$ implies that this is at least $1/8$, and the result follows.
\end{proof}

\begin{corollary} \label{2Dsubspaces}
Provided that $2\lambda^2\be+6\e<1/8$, every great circle contains a point that does not belong to~$\Gamma_\e$.
\end{corollary}

\begin{proof}
The previous Corollary \ref{2Dsubspaces}, applied to the subspace whose unit sphere is the great circle, gives us a point $y$ 
such that $\P^y_{i,j}[y_i/y_j\notin B^2B^{-2}]\geq 1/8$. Since the event in square brackets is invariant under positive scalar multiples, we may assume that $y$ is a unit vector and thus that it belongs to the great circle.

We showed earlier that if $z\in\Gamma_\e$, then $\P^z_{i,j}[z_i/z_j\in B^2B^{-2}]\geq 1-2\lambda^2\beta-6\e$. Hence, if $2\lambda^2\beta+6\e<1/8$, this implies that $y\notin\Gamma_\e$, and we are done.
\end{proof}

\section{Almost every point has an ``atypical'' image}\label{sec2.4}

In this section we want to show that there exists a subspace $X$ of linear dimension such that $X\cap S^{n-1} \subset \Delta^c= \{x\in S^{n-1}: \ \P^x_i[x_i\in A]\ge1-\be \}$, so that $\psi(X\cap S^{n-1})\subset \psi (\Delta^c) \subset \Gamma$. Indeed, we shall show that for an appropriate constant $\a>0$, almost all subspaces of dimension at most $\a n$ have this property. To this end, it will be sufficient to show that $\Delta$ has exponentially small measure. Note that
\[ \P[ x\in \Delta ]=\P [ \P^x_i[x_i \in A]<1-\be ]=\P [ \P^x_i[x_i \in B]\ge \be],
\]
where $B$ is, as before, the set \[\bigcup_m(B_m\cup(-B_m)),\] 
and $B_m=[s^{2mk-1},s^{2mk+1}]$ for each integer $m$. Let $\eta=s-1>0$ and as before let $\lambda=s^{2k}$. Then $B_m=[(1+\eta)^{-1} \lambda ^m, (1+\eta)\lambda^m]$. 

Let us say that a positive real number $t$ is an $\eta$-\emph{approximate power of} $\lambda$ if there exists an integer $m$ such that \[(1+\eta)^{-1}\lambda^m\leq t\leq(1+\eta)\lambda^m.\] 

For $\g \in [0,1]$ and $\xi\ge0$ define $\Delta_\g^\xi$ by 
\begin{align}\label{def: delta xi gamma}
    \Delta_\g^\xi=\Bigl\{x\in \R^n:\P^x_i[|x_i|\ \hbox{is a $\xi$-approximate power of}\ \lambda]\geq\g\Bigr\}.
\end{align} 
As mentioned before, we shall end up taking $\lambda=4$. For this reason, although $\Delta_\g^\xi$ depends on $\lambda$, we suppress this dependence in the notation. We shall be particularly interested in the set $\Delta_\be^\eta$, which, when restricted to $S^{n-1}$, is equal to the set $\Delta$ defined in \eqref{def: delta set}.

However, we shall also be interested in the set $\Delta_\be^0$, which we shall write simply as $\Delta_\be$. That~is, 
\begin{align}
    \Delta_\be=\Bigl\{x \in \R^n:\P_i^x[|x_i|\ \hbox{is a power of}\ \lambda]\geq\be\Bigr\}.
\end{align}

\begin{lemma} \label{approx}
If $y\in\Delta_\be^\eta$ then there exists $x\in\Delta_{\be(1-2\eta)}$ such that $\|x-y\|\leq\eta\|y\|$.
\end{lemma}

\begin{proof}
We are given that $\P_i^y[|y_i|\ \hbox{is an $\eta$-approximate power of}\ \lambda]\geq\be$. Let $J$ be the set of all $i$ such that $|y_i|$ is an $\eta$-approximate power of $\lambda$. For each $i\in J$ let $|x_i|$ be the nearest power of $\lambda$ to $|y_i|$ and let $x_i$ have the same sign as $y_i$. For each $i\notin J$ let $x_i=y_i$. Then $|x_i-y_i|\leq\eta|y_i|$ for $i\in J$, so, writing $P_J$ for the coordinate projection to $J$, we have that 
\[\|x-y\|^2=\frac 1n\sum_{i\in J}|x_i-y_i|^2\leq\eta^2\frac 1n\sum_{i\in J}|y_i|^2=\eta^2\|P_Jy\|^2\leq\eta^2\|y\|^2.\]
We now need a lower bound for $\|P_Jx\|^2/\|x\|^2$. We know that $\|P_Jy\|^2\geq\be\|y\|^2$, and also that $\|P_Jx\|^2-\|P_Jy\|^2=\|x\|^2-\|y\|^2$. We also have for each $i\in J$ that $(1+\eta)^{-2}y_i^2\leq x_i^2\leq(1+\eta)^2y_i^2$, which implies that $(1+\eta)^{-2}\|P_Jy\|^2\leq\|P_Jx\|^2\leq(1+\eta)^2\|P_Jy\|^2$. Therefore, 
\[\frac{\|P_Jx\|^2}{\|x\|^2}=\frac{\zeta\|P_Jy\|^2}{\|y-P_Jy\|^2+\zeta\|P_Jy\|^2}\]
for some $\zeta\in[(1+\eta)^{-2},(1+\eta)^2]$. The right-hand side is minimized when $\zeta=(1+\eta)^{-2}$, and then it is at least $\zeta\be\ge (1-2\eta)\be$, which finishes the proof of the lemma.
\end{proof}

Our next aim is to prove an upper bound for the volume of the $\eta$-expansion of  $\Delta _{\be(1-2\eta)}$, which by the above lemma contains $\Delta_\be^\eta$. We shall do this in a series of simple steps.

\begin{lemma} \label{simplex}
For every constant $C>1$, the number of sequences $(a_1,\dots,a_m)$ of positive integers that add up to at most $Cm$ is at most $(Ce)^m$.
\end{lemma}

\begin{proof}
For each sequence $a=(a_1,\dots,a_m)$ let $C_a$ be the unit cube of points $x\in\R^n$ such that $a_i-1\leq x_i<a_i$ for every $i$. Then if $a\ne b$, the unit cubes $C_a$ and $C_b$ are disjoint. Also, if $a$ consists of positive integers and $\sum_ia_i\leq Cm$, then the cube $C_a$ is contained in the convex hull of the points $0$ and $Cme_i$, where $e_1,\dots,e_m$ is the standard basis of $\R^m$. But this simplex has volume $(Cm)^m/m!\leq(Ce)^m$. This proves the result.
\end{proof}

\begin{corollary} \label{jensen}
Let $\lambda,\,C>1$ be real numbers and let $m$ be a positive integer. Then the number of positive integer sequences $(a_1,\dots,a_m)$ such that $\lambda^{a_1}+\dots+\lambda^{a_m}\leq Cm$ is at most $(e\log_\lambda C)^m$.
\end{corollary}

\begin{proof}
If $\lambda^{a_1}+\dots+\lambda^{a_m}\leq \lambda^am$, then by Jensen's inequality $\lambda^{(a_1+\dots+a_m)/m}\leq \lambda^a$, and hence $a_1+\dots+a_m\leq am$. Therefore, by Lemma \ref{simplex} the number of such sequences is at most $(ea)^m$. By the hypothesis we have $a=\log_\lambda C$, and the result follows.
\end{proof}

\begin{corollary} \label{lognet}
Let $\lambda>1$ be a real number, let $m$ be a positive integer, and let $\eta>0$. Let $\Omega$ be the set of all sequences $(x_1,\dots,x_m)$ such that $x_1^2+\dots+x_m^2\leq C^2m$ and each $|x_i|$ is a power of $\lambda$. Then there is an $\eta$-net of $\Omega$ of cardinality at most $(2e\log_\lambda(\lambda^2C/\eta))^m$.
\end{corollary}

\begin{proof}
Let $x\in\Omega$. For each $i$ such that $|x_i|\leq\eta/\lambda$, replace $x_i$ by $\lambda^{-t}\mathop{\mathrm{sign}}(x_i)$, where $t$ is chosen in such a way that $\eta/\lambda\leq \lambda^{-t}<\eta$, and let the resulting vector be $y$. Then $|x_i-y_i|\leq\eta$ for every $i$, so $\|x-y\|\leq\eta$.  Now let $\Omega'$ consist of all vectors $x\in\Omega$ such that each $|x_i|$ is equal to $\lambda^{a_i}$ for some integer $a_i$ with $a_i\geq-t$. We have just shown that $\Omega'$ is an $\eta$-net of $\Omega$.  

The number of points in $\Omega'$ with positive coordinates is equal to the number of integer sequences $(a_1, \dots ,a_m)$ such that each $a_i$ is at least $-t$ and $\lambda^{2a_1}+\dots+\lambda^{2a_m}\leq C^2m$. Rescaling, we see that is the number of positive-integer sequences $(a_1,\dots,a_m)$ such that $\lambda^{2a_1}+\dots+\lambda^{2a_m}\leq \lambda^{2(t+1)}C^2m$, which by Corollary \ref{jensen} is at most $(e(t+1+\log_\lambda C))^m$. Since there are $2^m$ possible choices of signs, the size of $\Omega'$ is at most $(2e(t+1+\log_\lambda C))^m$. Noting that $t\leq\log_\lambda(\lambda/\eta)=1+\log_\lambda(1/\eta)$, we obtain the result.
\end{proof}

The important thing about the bound above is that the number we raise to the power $m$ depends logarithmically on $\eta$. This shows that an $\eta$-net of $\Omega$ is \emph{much} smaller than an $\eta$-net of the full sphere of radius $C$. 

We shall need a lemma concerning the sizes of nets of unit balls. It is standard, but the version we give is less commonly used, so for convenience we include a proof. (The argument is essentially due to Rogers~\cite{rogers1957}.)

\begin{lemma}\label{net}
Let $X$ be an $n$-dimensional normed space with unit ball $B_X$ and let $\d>0$. If $n$ is sufficiently large, then $X$ contains a $\d$-net of $B_X$ of cardinality at most $2en\log( n) (1+\frac 1\d)^n$.
\end{lemma}

\begin{proof}
Let $\rho>0$ be a small real number to be chosen later. (It will in fact depend on $n$.) Then a standard volume estimate shows that there is an $\rho$-net of $B_X$ of size at most $(3/\rho)^n$. We shall now cover every point of this net with a union of balls of radius $\d-\rho$ in order to obtain our $\d$-net, and then we will optimize over $\rho$.

To do this, let $\zeta=\d-\rho$ and pick points $x_1,\dots,x_N$ uniformly at random from $(1+\zeta)B_X$. If $y$ is a point in the $\rho$-net, then the probability that $y$ is not within any of the balls of radius $\zeta$ about the $x_i$ is $(1-(\frac{\zeta}{1+\zeta})^n)^N\leq\exp(-N(\frac\zeta{1+\zeta})^n)$. Therefore, we are done as long as
\[\Bigl(\frac 3\rho\Bigr)^n\exp\Bigl(-N\Bigl(\frac\zeta{1+\zeta}\Bigr)^n\Bigr)<1,\]
which is satisfied if $N>n\log(\frac 3\rho)(1+\frac 1{\d-\rho})^n$. 

It can be checked that $1+\frac 1{\d-\rho}=(1+\frac 1n)(1+\frac 1\d)$ when $\rho=\d(\frac{\d+1}{n+\d+1})$. For this value of $\rho$ and for $n$ is sufficiently large, we have that $\log(\frac 3\rho)<\log(\frac{3n}\d)$, which is at most $\frac 32\log n$. We also have that $(1+\frac 1n)^n<\frac {4e}3$ when $n$ is sufficiently large, and putting these estimates together we find that we can take $N$ to be $2en\log n(1+\frac 1\d)^n$, as claimed.
\end{proof}

Next, we need a simple technical lemma about the largest proportion of the unit sphere of $L_2^n$ that can be covered by a ball of radius $\d$. 

\begin{lemma} \label{cone}
Let $B_\d (x)$ be a closed ball of radius $\d$ about a point $x$ in $L_2^n$. If $n$ is sufficiently large, then the probability that a random point of the unit sphere of $L_2^n$ lies in $B_\d (x)$ is at most $2\d^n$.
\end{lemma}

\begin{proof}
The intersection of $B_\d (x)$ with the unit sphere is a spherical cap, and the measure of the spherical cap is maximized when the centre $x$ of $B_\d (x)$ is a vector of norm $\sqrt{1-\d^2}$.

Define $C$ to be the set of all $y$ such that $\|y\| \le 1$ and $\bigl\|x-\frac y{\|y\|}\bigr\|\leq\d$. This is a convex hull of the spherical cap and the origin, and the proportion of its volume to the volume of the entire unit ball, is equal to the probability we are trying to estimate.

We are going to show that $B_\d (x)$ contains the set $C\setminus(1-2\d^2)C$. Indeed, we claim now that $B_\d (x)$ contains all points $y$ such that $\bigl\|x-\frac y{\|y\|}\bigr\| \leq\d$ and $1 \geq \|y\|\geq 1-2\d^2$. By the convexity of $B_\delta(x)$ it is sufficient to prove this when $\|y\|=1-2\d^2$. The first assumption on $y$ implies that
\[\|x\|^2-\frac{2\langle x,y\rangle}{n\|y\|}+1\leq\d^2,\]
and therefore, since $\|x\|^2=1-\d^2$, that
\[\langle x,y\rangle\geq(1-\d^2)n\|y\|.\]
This implies that
\begin{align*}
\|x-y\|^2&= \|x\|^2+\|y\|^2-\frac{2\langle x,y\rangle}{n}\\
&\leq 1-\d^2+(1-2\d^2)^2-2(1-\d^2)(1-2\d^2)\\
&=\d^2,
\end{align*}
which proves the claim.

We have therefore shown that $B_\d (x)$ contains the set $C\setminus(1-2\d^2)C$. Now, since $(1-2\d^2)C$ has volume $(1-2\d^2)^n$ times that of $C$, if $n$ is sufficiently large, then $B_\d (x)$ contains at least half of $C$. The result follows, since the volume of $B_\d (x)$ is $\d^n$ times that of the unit sphere of $L_2^n$.
\end{proof}

Now let $y$ be a vector in $L_2^n$ supported on $J\subset \left\lbrace 1, \ldots , n\right\rbrace$ of cardinality $m$ and satisfying the inequality $\|y\|^2\geq\be(1-2\eta)$. Again let $P_J$ be the coordinate projection to the set $J$ and define
\[V_y = \{ x\in S^{n-1}: \, P_J x=y \} .\]

 We will next obtain an upper bound for the spherical volume of $(V_y)_\e$, which is the $\e$-expansion of $V_y$.

\begin{lemma} \label{ringvolume}
Let $\d>\eta>0$. Then when $n$ is sufficiently large, the probability that a random unit vector belongs to $(V_y)_\eta$ is at most $4e \d ^n n\log n\Bigl(1+\frac{\sqrt{1-\be(1-2\eta)}}{\d-\eta}\Bigr)^{n-m}$.
\end{lemma}

\begin{proof}
If we cover $V_y$ by $N$ balls of radius $\d-\eta$, then the balls of radius $\d$ with the same centres cover $(V_y)_\eta$, so by Lemma \ref{cone} the probability that a random unit vector lies in $(V_y)_\eta$ is at most $2N\d^n$. But $V_y$ is an $(n-m)$-dimensional sphere of radius at most $\sqrt{1-\be(1-2\eta)}$, hence Lemma \ref{net} implies that it can be covered by at most $N=2en\log n\Bigl(1+\frac{\sqrt{1-\be(1-2\eta)}}{\d-\eta}\Bigr)^{n-m}$ balls of radius $\d-\eta$. This implies the result.
\end{proof}

\begin{theorem} \label{badvolume}
The probability that a random unit vector belongs to $(\Delta_{\be(1-2\eta)})_\eta$ is exponentially small.
\end{theorem}

\begin{proof}
Lemma \ref{approx} tells us that the set $\Delta_{\be(1-2\eta)}$ is an $\eta$-net of the set of unit vectors in $\Delta_\be^\eta$. Moreover, the same is true if we restrict to vectors of norm at most $1+\eta\leq 2$. For each $x\in\Delta_{\be(1-2\eta)}$ there is a set $J\subset\{1,2,\dots,n\}$ such that $\mu_x(J)\geq\be(1-2\eta)$ and $|x_i|$ is a power of $\lambda$ for every $i\in J$. If $|J|=m$, then Corollary \ref{lognet} implies that there is an $\eta$-net of size at most $\left(2e\log_\lambda \left(\sqrt{\frac{2n}{m}} \frac{\lambda^2}{\eta}\right) \right)^m$ of the set of vectors $y$ such that $|y_i|$ is a power of $\lambda$ for every $i\in J$, $y_i=0$ for $i\notin J$, and $\sum_iy_i^2\leq 2n$. 

Every unit vector in $\Delta_{\be(1-2\eta)}$ lies in $V_y$ for some such $J$ and $y$.
Therefore, summing over all $J$ and all $y$ in an $\eta$-net for each $J$ and applying Lemma \ref{ringvolume}, we find that the probability that a random unit vector belongs to $(\Delta_{\be(1-2\eta)})_\eta$ is at most

\[\sum_{m=1}^n\binom nm\Bigl(2e\log_\lambda\Bigl(\sqrt{\frac{2n}{m}} \frac{\lambda^2}{\eta}\Bigr)\Bigr)^m4e\d^nn\log n\Bigl(1+\frac{\sqrt{1-\be(1-2\eta)}}{\d-\eta}\Bigr)^{n-m}.\]
Now let us set $\lambda=4$. Using the upper bound $\binom nm\leq(en/m)^m$ and setting $\theta=m/n$, we can bound the previous expression above by
\[4en\log n\sum_{m=1}^n\Bigl(2e^2\d\frac 1\theta\log_4\Bigl(\frac{16\sqrt{2}}{\eta\sqrt{\theta}}\Bigr)\Bigr)^{\theta n}\Bigl(\d+\frac{\d\sqrt{1-\be(1-2\eta)}}{\d-\eta}\Bigr)^{(1-\theta)n}.\]
To prove that this is exponentially small, it is sufficient to show that  
\begin{align}\label{eq:prob}
    \Bigl(\frac{2e^2}{\log 4} \frac{\d}{\theta}\log\Bigl(\frac{16\sqrt{2}}{\eta\sqrt{\theta}}\Bigr)\Bigr)^{\theta}\Bigl(\d+\frac{\d\sqrt{1-\be(1-2\eta)}}{\d-\eta}\Bigr)^{1-\theta}
\end{align}
is bounded above by a constant less than 1 as $\theta$ varies. First of all, note that in \S\ref{sec2.3} we obtain the following bounds on parameters: $\eta\le 10^{-5}$ (since $\eta=s-1$ and we need $\tau=1-s^{-2}\le 10^{-4}$), $\e \le \tau^2 $ and $\beta< \frac{1}{32}\left(\frac{1}{8} - 6\epsilon \right)$. We need moreover, that $\d>\eta$. Let us note that the above expression is a decreasing function of $\beta$ and since $\e \le 10^{-8}$ we can take $\be= \frac{1}{257}$.

To begin with, we shall show that there exist constants for which the result holds and then we shall choose some particular values to obtain an upper bound for the maximum. Let us consider the condition $\d+\frac{\d\sqrt{1-\be(1-2\eta)}}{\d-\eta}<1$. For $\eta = c\d$ this becomes
\[ \d +\frac{\sqrt{1-\be(1-2\eta)}}{1-c}<1,\]
and we can choose $c$ such that $\frac{\sqrt{1-\be(1-2\eta)}}{1-c}$ is less than $1-\be/4$, and then if $\delta \le \be/8$ the inequality holds. For the first part of expression (\ref{eq:prob}), if we assume that $\theta\ge \sqrt{\delta}$ (and again take $\eta =c \delta$) we have that
\begin{align*}
    \Bigl(\frac{2e^2}{\log 4} \frac{\d}{\theta}\log \Bigl(\frac{16\sqrt{2}}{\eta\sqrt{\theta}}\Bigr)\Bigr)^{\theta} &\le \left( 2e^2 \sqrt{\delta }\Bigl( \log (16\sqrt{2})  + \log \frac{1}{c \delta} + \frac{1}{4}\log\frac{1}{\delta} \Bigr)\right)^\theta \\ &\le 2e^2 \sqrt{\delta} \left( C_1 + \frac{5}{4} \log \frac{1}{\delta} \right)^\theta,  
\end{align*}
where $C_1$ is an absolute constant. Hence, we can choose $\delta$ small enough such that this expression is at most, say, $\frac{9}{10}$. 

If $\theta< \sqrt{\delta}$, then we need to consider the whole expression (\ref{eq:prob}). To begin with note that
\begin{align*}
    \Bigl(\frac{2e^2}{\log 4} \frac{\d}{\theta}\log \Bigl(\frac{16\sqrt{2}}{\eta\sqrt{\theta}}\Bigr)\Bigr)^{\theta} &\le \Bigl( 2e^2 \frac{\delta}{\theta} (C_1  + \frac{5}{2}\log\frac{1}{\theta} )\Bigr)^\theta \le \left( \frac{C_2 \delta \log \frac{1}{\theta}}{\theta} \right)^\theta \\ &\le (C_2 \delta)^\theta \left(\frac{1}{\theta^2}\right)^\theta \le \left(\frac{1}{\theta^2}\right)^\theta,
\end{align*}
for $\delta <1/C_2$. Moreover, we have that $\log \left( (\frac{1}{\theta})^{2\theta}\right) =2\theta \log \frac{1}{\theta} \le 2 \sqrt{\theta} \le 2 \delta^{1/4}$. Therefore we can estimate $\left(\frac{1}{\theta^2}\right)^\theta$ by $1+4 \delta ^{1/4}$. Recalling that the right hand side part in (\ref{eq:prob}) is at most $(1-\frac{\beta}{8})^{1-\theta} \le (1-\frac{\beta}{8})^{1-\sqrt{\delta}}$, we deduce that the expression (\ref{eq:prob}) is less than $1$ if we choose $\delta$ such that $(1+4\delta^{1/4})(1-\frac{\beta}{8})^{1-\sqrt{\delta}} <1$. Finally we choose the smallest $\delta$, so that it fulfils all the inequalities and hence the result follows.

One can check that if we choose $\eta =10^{-12}, \delta= 10^{-6}$ and $\beta=\frac{1}{257}$, then the maximum is  at most 0.9982. Therefore, the desired probability is at most $4en^2 \log n (0.9982)^n$.

\end{proof}

Recall the set $\Delta=\Delta_\beta^\eta $ defined in \eqref{def: delta set}. It follows that

\begin{corollary} \label{randomsubspace}
There exists $\a>0$ such that if $n$ is sufficiently large, then the probability that a random subspace $X$ of dimension at most $\a n$ contains a vector $x\in \Delta$, is exponentially small. 
\end{corollary}

\begin{proof} 
Let $\sigma<1$ be such that the probability that a random unit vector belongs to $(\Delta_{\be(1-2\eta)})_\eta$ is at most $\sigma^n$ when $n$ is sufficiently large. Now choose $\a>0$ such that $(1+\frac 1\eta)^\alpha<\sigma^{-1}$. Then for sufficiently large $n$, the unit sphere of any subspace of dimension at most $\alpha n$ has an $\eta$-net of size $\tau^{-n}$ for some $\tau$ with $\tau^{-1}<\sigma^{-1}$. If we take such a net and rotate it randomly, then the probability that any element of the net lands within $\eta$ of $\Delta_{\beta(1-2\eta)}$ is exponentially small. 

By Lemma \ref{approx} we have that $\Delta^\eta_\beta \subset (\Delta_{\be(1-2\eta)})_\eta$ and hence it follows that the probability that a random subspace intersects $\Delta$ is exponentially small.
\end{proof}

\begin{remark}
For the particular choice of constants made in the proof of Corollary \ref{badvolume}, the condition we obtain in Corollary \ref{randomsubspace} is
\[ 0.9982 (4en^2 \log n )^{1/n} < \sigma .\]

The left hand side is a decreasing function of $n$ with limit $0.9982$, so for $n$ large enough the left hand side is less than $0.999$. We then have that

\[ \alpha < \frac{- \ln (0.999) }{\ln(1+10^{10})} \approx 4.345\times 10^{-5}. \]
 Hence, for $n$ large enough we can take $\alpha =4.3 \times10^{-5}$ in the construction. 

\end{remark}

\section{Conclusion}

\iftrue
\else
We have now proved our main theorem, Theorem \ref{main}, but since the steps of the argument are somewhat scattered, let us briefly recall them. We defined a continuous odd bijection $\v:\R^{n}\to\R^{n}$ and a normalized version $\psi:S^{n-1}\to S^{n-1}$. In \eqref{def: delta set} we defined a set $\Delta$ and in the last section $\S$\ref{sec2.4} we proved that a random $\lfloor\a n\rfloor$-dimensional subspace lies in $\Delta^c$ with exponentially high probability. 

We also defined a set $\Gamma\subset S^{n-1}$, see \eqref{def: gamma set}, and observed that $\psi(\Delta^c)\subset\Gamma$. Section \ref{sec2.3} was devoted to the proof that no great circle is contained in the expansion $\Gamma_\e$. 

Therefore, if we pick a random $\lfloor\a n\rfloor$-dimensional subsphere $X$, then with exponentially high probability we have that $X\cap S^{n-1}\subset\Delta^c$, and therefore that $\psi(X\cap S^{n-1})\subset\psi(\Delta^c)\subset\Gamma$, from which it follows that $\psi(X)_\e$ contains no great circle. This gives Theorem \ref{main}. 
\fi

In this final section we briefly explain our original motivation for this result. We were interested in the following question of Milman, concerning a possible strengthening of Dvoretzky's theorem.

\begin{question} Let $k$ be a positive integer, let $C\geq 1$ and let $\e>0$. Does there exist $n$ such that if $X=(\R^n,\vertiii{\cdot})$ is any normed space such that $\|x\|\leq \vertiii{x}\leq C\|x\|$ for every $x\in X$, then $X$ has a subspace of dimension $k$ that is $(1+\e)$-complemented and has Banach-Mazur distance at most $1+\e$ from $\ell_2^k$? 
\end{question}

For the reader unfamiliar with the terminology, the \emph{Banach-Mazur distance} between two isomorphic normed spaces $X$ and $Y$ is the infimum of $\|T\|\|T^{-1}\|$ over all linear isomorphisms $T:X\to Y$, and a subspace $U\subset X$ is $\alpha$-\emph{complemented} if there is a projection $P:X\to U$ with $\|P\|\leq\a$.

We initially attempted to obtain a positive answer to a stronger question, namely the following.

\begin{question} \label{strongmilman} Let $k$ be a positive integer, let $C\geq 1$ and let $\e>0$. Does there exist $n$ such that if $X=(\R^n,\vertiii{\cdot})$ is any normed space such that $\|x\|\leq \vertiii{x}\leq C\|x\|$ for every $x\in X$, then $X$ has a subspace $U$ of dimension $k$ such that there exists $\be$ with $\be\|u\|\leq \vertiii{u}\leq\be(1+\e)\|u\|$ for every $u\in U$ and such that $\|P\|\leq 1+\e$, where $P$ is the orthogonal projection from $X$ to $U$?
\end{question}

In relation to this conjecture, we identified a class of points $x\in X$ that we called $\e$-\emph{good}: these are points $x$ such that the orthogonal projection to the 1-dimensional space spanned by $x$ has norm at most $1+\e$. The second question turns out (by a not very difficult argument) to be equivalent to asking whether for sufficiently large $n$ there is a $k$-dimensional subspace consisting entirely of $\e$-good points. 

A number of examples have led us to believe that for any normed space satisfying the conditions of the question, the set of $\e$-good points is ``genuinely $cn$-dimensional" for some positive constant $c$ that depends on $C$ and $\e$. (We have not formulated a suitable conjecture, but one possible definition of a set $X\subset S^{n}$ such that $X=-X$ being ``at least $m$-dimensional" is that when we regard $X$ as a subset of projective $n$-space, it is not homotopic to a subset of dimension less than $m$.) Also, a point that is close to an $\e$-good point is $2\e$-good. Therefore, a positive answer to the question would follow if one could show that the $\e$-expansion of a ``genuinely high-dimensional" subset of the sphere contains a $k$-dimensional linear subsphere. 

However, the main result of this paper shows that this is false. In a follow-up paper we shall show that the answer to Question \ref{strongmilman} is also negative. However, while the two results arose from the same line of thought, the constructions are somewhat different, and neither result directly implies the other.

\iftrue
\else
As explained at the end of the introduction, this also gives a counterexample to Question \ref{lastversion}. However, there is a small technical issue in that our scalar function $\v:\R\to\R$ used in the construction is not differentiable at zero, and hence the function $\psi:S^{n-1}\to S^{n-1}$ is not differentiable at any point with a coordinate equal to zero.

This can be dealt with in the obvious way -- by approximating $\v$ by a function that is differentiable everywhere and equal to $\v$ except inside a very small interval about 0. Unfortunately, there does not seem to be a quick and easy argument that this change does not give rise to a great circle that consists entirely of $\delta$-good$^*$ points. However, it is straightforward to make small adjustments to our construction, the main one of which is to enlarge very slightly the set $B$ so that it contains an interval about 0, and to show that the final result is indeed robust in the appropriate sense.

Thus, a question that can be thought of as a suitable Lipschitz version of Milman's original question has a negative answer. However, because the derivatives of the function $\psi$ at two points $x$ and $y$ can be very different even when $x$ and $y$ are close, when we compose it with a weighted $\ell_2$ norm as described in \S\ref{connection}, we obtain a function that is not a norm, so we do not obtain a counterexample to Question \ref{strongmilman} (or equivalently Question \ref{egoodquestion}). To obtain such a counterexample, one would need much better control over the second derivative of $\psi$.

\section{A norm with few $\e$-good points such that the standard basis is symmetric.} 

The example we gave earlier of a norm such that the set of $\e$-good points has small measure was just the Euclidean norm in a non-standard position. From the perspective of Milman's question, this is an unsatisfactory example, since it is isometric to the usual Euclidean norm. In particular, we have not ruled out that for any norm that is $C$-Euclidean, there is some position (that is, invertible linear transformation of the norm) such that almost all points are $\e$-good, which would straightforwardly imply a positive answer to Milman's question.

In this short section we present strong evidence that such a result cannot be obtained, by giving an example of a $C$-Euclidean norm such that the standard basis is 1-symmetric and the measure of $\e$-good points is exponentially small. We do not prove that the same is true for all norms that are equivalent to this one, but it seems highly unlikely that a non-standard position of a highly symmetric unit ball would have far more good points than the standard one.

The norm we take is the mixed $\ell_2$/$\ell_\infty$ norm on $\R^n$ defined by the formula
\[\|x\|=\max _{|I|=m}|P_I x|,\] 
where $P_I$ is the coordinate projection to the set of indices $I$. If $m=cn$, then it is simple to see that $|x|\geq\|x\|\geq c^{1/2}|x|$ for every $x$, so this norm is $c^{-1/2}$-Euclidean.

\begin{lemma}
The measure of the set of $\epsilon$-good points $x\in S^{n-1}$ with respect to the norm $\| \cdot \|$ is at most $(12 \e^{\frac{1-\lambda}{2}})^n \frac{1}{(1-\sqrt{\e}-\lambda)\sqrt{n}}$, where $\lambda=\frac{m+1}{n}$.
\end{lemma}

\begin{proof}
Let $x$ be an $\e$-good point and suppose that $x_1\geq x_2\geq\dots\geq x_n\geq 0$. Let $y$ be the vector with coordinates $(x_1,x_2,\dots,x_m,x_m,\dots,x_m)$, and note that $\|y\|=\|x\|$. Since $x$ is $\e$-good, $y$ is not a witness for $x$ being $\e$-bad, and therefore
\[\langle x,y\rangle\leq(1+\e)|x|^2.\]
Expanding both sides, we find that
\[\sum_{i\leq m}x_i^2+\sum_{i>m}x_ix_m\leq(1+\e)\sum_ix_i^2\]
and therefore that
\[\sum_{i>m}x_i(x_m-x_i)\leq\e\sum_{i}x_i^2.\]

Let $A=\{i>m:x_i\geq x_m/2\}$ and let $u=P_Ay$. In other words, $u_i=x_m$ if $x_i\geq x_m/2$ and 0 otherwise. Then $(x_i-u_i)^2\leq x_i(x_i-u_i)$ for every $i$. If we also let $v$ be the coordinate projection of $x$ to the set of coordinates greater than $m$, then this implies that $|v-u|^2\leq\e |x|^2 =\e$.

Hence, for every $\e$-good unit vector $x$ we can find a set $J$ that consists of $n-m$ coordinates, a positive real number $\lambda$, and a vector $y$ such that $|x-y|\leq\sqrt\e$ and $y_i\in\{0,\lambda,-\lambda\}$ for every $i\in J$. 

It is therefore sufficient to prove that the $\sqrt\e$-expansion of the set of such $y$ intersects the unit sphere in a set of small measure.

To do this, let us first fix a set $J$ and a subset $A\subset J$ and a set of signs $\e_i=\pm 1$, one for each $i\in A$, and consider the set of all $y$ such that $\e_iy_i$ is constant on $A$ and $y_i=0$ on $J\setminus A$. This is a subspace of $\R^n$ of dimension $m+1$, so by standard estimates \cite{Artstein2002} its $\sqrt\e$-expansion intersects the unit sphere in a set of measure at most 
\[ \mu(( E_{m+1})_{\sqrt{\e}}) \simeq \frac{1}{\sqrt{n\pi}} \frac{\sqrt{\lambda(1-\lambda)}}{(1-\lambda)-\sin^2(\sqrt{\e})}e^{-\frac{n}{2}u(\lambda,\e)} \le \frac{1}{\sqrt{n\pi}} \frac{1}{cos^2 (\sqrt{\e})-\lambda} e^{-\frac{n}{2}u(\lambda,\e)},
\]
where $\lambda = \frac{m+1}{n}$ and $u(\lambda,\e)= (1-\lambda) \log \frac{1-\lambda}{\sin^2 (\sqrt{\e})} + \lambda \log \frac{\lambda}{\cos ^2( \sqrt{\e})} $. This we can estimate from below \[u(\lambda,\e) =\log \left( \frac{1-\lambda}{\e} \left(\frac{\lambda}{1-\lambda}\right)^\lambda (\tan \sqrt{\e})^{2\lambda}\right) \ge \log \left(  \e ^{\lambda-1} \frac{\lambda^\lambda }{(1-\lambda)^{\lambda-1}} \right). \]

Hence we have
\begin{align*}
    \mu(( E_{m+1})_{\sqrt{\e}}) &\le \frac{1}{\sqrt{n}(1-\sqrt{\e} - \lambda)}\left( \e ^{\lambda-1} \frac{\lambda^\lambda }{(1-\lambda)^{\lambda-1}}\right)^{-n/2} \\ &\le \frac{1}{\sqrt{n}(1-\sqrt{\e} - \lambda)}\left( \frac{\e ^{1-\lambda}}{(1-\lambda)^{1-\lambda} \lambda ^\lambda} \right)^{n/2}.
\end{align*}

The number of ways of choosing $A,J$ and the signs is $\binom n{\lambda n}3^{(1-\lambda)n}$. 
Therefore, for a fixed $\lambda=(m+1)/n$ we get in total that the measure of good points is at most
\begin{align*}
    \binom n{\lambda n}3^{(1-\lambda)n} \frac{1}{\sqrt{n}(1-\sqrt{\e} - \lambda)} & \left( \frac{\e ^{1-\lambda}}{(1-\lambda)^{1-\lambda} \lambda ^\lambda} \right)^{n/2} \\ &\le \left(\frac{e}{\lambda} \right)^{\lambda n}\frac{1}{\sqrt{n}(1-\sqrt{\e} - \lambda)}\left( \frac{\e ^{1-\lambda}}{(1-\lambda)^{1-\lambda} \lambda ^\lambda} \right)^{n/2} 3^{(1-\lambda)n} \\ 
    &\le \frac{1}{\sqrt{n}(1-\sqrt{\e} - \lambda)} 12^n \e^{(1-\lambda)n/2},
\end{align*} 
since the function $x^{3x/2}(1-x)^{(1-x)/2}$ is bounded below for all $x\in [0,1]$. Hence, the measure of $\e$-good points is at most $(12 \e^{\frac{1-\lambda}{2}})^n \frac{1}{(1-\sqrt{\e}-\lambda)\sqrt{n}}$, and for a fixed epsilon  and $\lambda< \min \{ 1-2\sqrt{\e}, 1-2\frac{\log 12}{ \log \frac{1}{\e}} \}$, this tends to $0$ as the dimension increases. 
\end{proof}

\fi

\nocite{*}
\bibliography{bibfile} 
\bibliographystyle{amsplain}

\end{document}